\documentclass[12pt,reqno]{amsart}
\usepackage[margin=1in]{geometry}
\makeatletter
\def\section{\@startsection{section}{1}%
	\z@{.7\linespacing\@plus\linespacing}{.5\linespacing}%
	{\bfseries\normalfont\scshape
		\centering
}}
\def\@secnumfont{\bfseries}
\makeatother
\usepackage{hyperref}
\usepackage{amsmath,amsfonts,amsthm,txfonts}
\usepackage{dsfont}

\usepackage{graphicx,amssymb}
\usepackage{amsmath,amssymb,amsfonts}
\usepackage{hyperref}
\usepackage{amscd}
\usepackage{amsthm}
\usepackage{xcolor}
\usepackage{enumitem}
\numberwithin{equation}{section}
\newtheorem{Thm}{Theorem}[section]
\newtheorem{Lem}{Lemma}[section]
\newtheorem{Pro}{Proposition}[section]
\newtheorem{Rem}{Remark}[section]

\newtheorem{Def}{Definition}[section]

\usepackage{cases} 

\allowdisplaybreaks

\begin{document}
\title[Optimal Control of the Landau-Lifshitz-Gilbert Equation]
{Optimal Control of  the 2D Landau-Lifshitz-Gilbert Equation with \\ \vspace{.1in}  Control Energy in Effective Magnetic Field }
\author[Sidhartha Patnaik and  Sakthivel Kumarasamy ]{Sidhartha Patnaik and Sakthivel Kumarasamy$^*$}
\address{Department of Mathematics \\
	Indian Institute of Space Science and Technology (IIST) \\
	Trivandrum- 695 547, INDIA}
\email{sidharthpatnaik96@gmail.com, pktsakthi@gmail.com, sakthivel@iist.ac.in}
\curraddr{}
\begin{abstract}
		The optimal control of magnetization dynamics in a ferromagnetic sample at a microscopic scale is studied. The dynamics of this model is governed by the Landau-Lifshitz-Gilbert equation on a two-dimensional bounded domain with the external magnetic field (the control) applied through the effective field. We prove the global existence and uniqueness of a regular solution in $\mathbb S^2$ under a smallness condition on control and initial data. We establish the existence of optimal control and derive a first-order necessary optimality condition using the Fr\'echet derivative of the control-to-state operator and adjoint problem approach. 
\end{abstract}
\thanks{$^*$Corresponding author: Sakthivel Kumarasamy}
\thanks{The work of the second author is supported by the National Board for Higher Mathematics, Govt. of India through the research grant No:02011/13/2022/R\&D-II/10206.}
\maketitle	

\section{Introduction}	
The model of magnetization dynamics representing energy interactions between magnetic materials and the effect of an applied external magnetic field on magnetization dynamics was obtained by L.D. Landau and E.M. Lifshitz (\cite{LLEL}). By introducing dissipation phenomenologically, T.L. Gilbert  (\cite{TLG})  modified the Landau-Lifshitz equation. The current paper discusses the optimal control of magnetization dynamics in ferromagnetic material governed by the Landau-Lifshitz-Gilbert (LLG) equation. The control problem has numerous physical applications, including magnetic sensors and data storage devices(\cite{JMH},\cite{Pr}). In these applications, it is of utter importance that we would precisely control the magnetization process with the help of some applied magnetic field. Another crucial application of magnetization dynamics in medical science is magnetic nanoparticle hyperthermia, which is a cancer treatment that involves induced heating of nanoparticles subjected to the tumor with the help of an alternating magnetic field (see, \cite{MHLD}).

Let $\Omega$ be a bounded smooth domain in $\mathbb{R}^2$ occupied by some ferromagnetic material. Suppose $M:\Omega\times[0,T]\to\mathbb R^3$ represents the magnetization vector field. Below the Curie temperature in ferromagnetic materials, the magnitude of magnetization stays constant throughout the domain, that is $|M|=M_s$, where $M_s$ is the saturated magnetization. The normalized magnetization $m=\frac{M}{M_s}$ belongs  to $\mathbb S^2,$ the unit sphere in $\mathbb R^3.$  For $(x,t)\in\Omega\times[0,T],$ the evolution of $m$ is described  by the LLG equation
$$m_t= \gamma m \times \mathcal{E}_{eff}(m)- \alpha\gamma m \times (m \times \mathcal{E}_{eff}(m)),$$
where $\times$ denotes the cross product in $\mathbb{R}^3,$  $\mathcal{E}_{eff}:\Omega\times[0,T]\to\mathbb R^3$ represents the effective field, $\alpha>0$ is called the Gilbert damping constant, and $\gamma$ denotes the gyromagnetic factor.  Further, the effective field is given by $\mathcal{E}_{eff}(m)=-\nabla_m \mathcal{E}(m)$, where the micromagnetism energy $\mathcal{E}$ governs various energy interactions within the ferromagnet specified by
$\mathcal{E}=\mathcal{E}_{ex}+\mathcal{E}_{an}+\mathcal{E}_{me}+\mathcal{E}_{d}+\mathcal{E}_{a},$
where $\mathcal{E}_{ex}$ is exchange energy, $\mathcal{E}_{an}$ is anisotropy energy, $\mathcal{E}_{me}$ is magnetoelastic energy, $\mathcal{E}_{d}$ is demagnetization field and $\mathcal{E}_{a}$ is the external magnetic field. For more details about energy interactions we are referring to \cite{WFB}. In this work, we have only considered the exchange energy and the external magnetic field.

In ferromagnetic materials, the individual atomic magnetic moments will attempt to align all neighboring atomic magnetic moments with themselves in the same direction due to exchange interaction. This deviation from their equilibrium state causes an addition in exchange energy. Hence, if we assume the external magnetic field to be the function $u:\Omega\times[0,T]\to\mathbb R^3$ and the energy field to be of a pure isotropic type, then the micromagnetism energy is given by
$$\mathcal{E}=\frac{1}{2} \int_\Omega |\nabla m|^2 dx- \int_\Omega u\cdot m\ dx.$$
If we consider the magnetic fields associated with these energies, then the effective field is given by
$\mathcal{E}_{eff}(m)=\Delta m+u.$
For a detailed summary of the model and physical meaning of the energies, we refer to \cite{MKAP}. 

In this paper,  we consider the optimal control problem of minimizing the  objective functional $\mathcal J:\mathcal M \times \mathcal{U}_{ad} \to \mathbb{R}^+$ defined as
\begin{eqnarray}\label{CF}
	\lefteqn{\mathcal J(m,u):= \frac{1}{2} \int_0^T\int_{\Omega} |m(x,t)-m_d(x,t)|^2 dx \ dt + \frac{1}{2} \int_{\Omega} |m(x,T) -m_\Omega(x)|^2\ dx }\nonumber\\
	&&\hspace{0.7in}+\frac{1}{2} \int_0^T\int_{\Omega} |u(x,t)|^2 \ dx \ dt+ \frac{1}{2} \int_0^T\int_{\Omega} | \nabla u(x,t)|^2 dx \ dt 
\end{eqnarray}
subject to the magnetization $m:\Omega\times[0,T]\to\mathbb R^3$ solves the following nonlinear Landau-Lifshitz-Gilbert equation with initial data $m_0$ and vanishing Neumann boundary condition:
\begin{equation}\label{NLP}
	\begin{cases}
		m_t=\gamma m \times (\Delta m +u)- \alpha\gamma m \times (m \times (\Delta m +u)),\ \ \ (x,t) \in \Omega_T:=\Omega\times (0,T], \\
		
		\frac{\partial m}{\partial \eta}=0, \ \ \ \ \ \ \ \ (x,t) \in \partial\Omega_T:= \partial \Omega\times [0,T],\\
		m(\cdot,0)=m_0 \ \ \text{in} \ \Omega,
	\end{cases}	
\end{equation}
where $\eta$ is the outward unit  normal vector to the boundary $\partial \Omega$ and $u$ is the external magnetic field. Here, we have assumed the desired evolutionary magnetic moment $m_d:\Omega\times[0,T]\to\mathbb R^3$ to be in $L^2(0,T;L^2(\Omega))$ and the final time target moment $m_{\Omega}:\Omega\to\mathbb R^3$ to be in $L^2(\Omega)$. Hereafter, for simplicity, we  set the parameters $\alpha=1$ and $\gamma=1.$   Further, throughout the paper, we  assume that the initial data $m_0:\Omega\to\mathbb R^3$ satisfies the following conditions
\begin{equation}\label{IC}
	m_0 \in H^2(\Omega),\ \ \frac{\partial m_0}{\partial \eta}=0 \ \ \text{on} \ \partial \Omega,\ \ |m_0|=1.	
\end{equation}

The control problem is motivated by the study of optimizing the switching processes in ferromagnets, the magnetic hard drive, where the external magnetic field serves as the control input responsible for the writing and reading phenomena(\cite{THL}). Mathematically, this can be interpreted as searching for an external magnetic field $\widetilde u$ and its corresponding magnetization vector field $\widetilde m$ such that the desired magnetization evolution $m_d$ and a target moment $m_\Omega$ can be attained with the least amount of control belonging to a suitable class of admissible external magnetic field, while the optimal pair of magnetic fields $(\widetilde m, \widetilde u)$ solves the LLG equation \eqref{NLP}.

In the absence of the external magnetic field $u(x,t),$ several works are available for the solvability of \eqref{NLP}. For example, the global existence of a weak solution for \eqref{NLP} and its non-uniqueness was proved in \cite{FAAS}. The authors \cite{BLGMCH} proved the global existence of a weak solution in a $m$-dimensional manifold and established a relation between harmonic maps and the solutions of the Landau-Lifshitz equation. 
For more results on a weak solution, one can also refer to \cite{ZJYW},\cite{AV}. The local time existence of a regular solution for a bounded domain of $\mathbb{R}^3$ was investigated in \cite{GCPF}, and  they also discussed the global existence and uniqueness of regular solutions under a smallness condition on the initial data in the 2D case. 
Apart from the literature on the well-posedness of \eqref{NLP} in the absence of an external magnetic field, very limited articles are available for the  control problems of \eqref{NLP}. The article \cite{FAKB} studies the optimal control type problems with the LLG equation as the state equation, and a necessary optimality system is derived when the  magnetization is constant in space, which eventually leads to an optimization problem constraint by an ordinary differential equation. The paper \cite{TDMKAP}, which is closely related to our work, discussed the optimal control of the 1D LLG equation and analyzed the numerical solution for this problem. Further, for the results related to controllability of the Landau-Lifshitz equation, we refer the readers to \cite{SAGCSLCP},\cite{GCSL2},\cite{AC}.

The main contributions of this paper are explained as follows. We proved the global solvability of the 2D LLG equation \eqref{NLP} with  space-time dependent external magnetic field, studied the optimal control of this problem and derived an optimality condition.

{\bf Global Solvability of \eqref{NLP}.} The local in-time existence of regular solution $m \in L^2(0,{\tilde T};H^3(\Omega))  \cap C([0,{\tilde T}];H^2(\Omega))$ both for 2D and 3D LLG equation with effective field $\mathcal E_{eff}(m)$ generated only by $m,$ and without  external magnetic field was proved in \cite{GCPF}.  Besides, for the 2D case with $\mathcal E_{eff}(m)=\Delta m,$ the authors studied the global existence and uniqueness of regular solutions under the smallness condition $\|\nabla m_0\|_{L^2(\Omega)}\leq \delta,$ where $\delta$ is sufficiently small. In the context of well-posedness of \eqref{NLP}, the current paper generalizes \cite{GCPF} to the case where the effective field $\mathcal E_{eff}(m)$ is modified by including the external magnetic field $u(x,t),$ which arises as a cross-product with the magnetic moment. By the method inspired in \cite{GCPF},  we first prove the local time existence of regular solutions of the 2D model \eqref{NLP} when $m_0$ satisfies \eqref{IC} and control  $u \in L^2(0,{T};H^1(\Omega)).$   It is known that the 3D Navier-Stokes equation admits a unique, strong solution  (\cite{PCCF}, Theorem 9.3) for small enough initial data and source term. In the spirit of \cite{GCPF,PCCF},  we extend the regular solution  for all time $t\in[0,T]$  under the assumption that the sum $\|\nabla m_0\|^2_{L^2(\Omega)} + \|u\|^2_{L^2(0,T;L^2(\Omega))} $ is sufficiently small and proven the uniqueness result. We give a detailed proof of these results (Theorem \ref{LOCAL} and Theorem \ref{GLOBAL}) to justify the necessary adaptation of the external magnetic field $u(x,t)$ occurring as a semilinear form in \eqref{NLP} and obtain optimal assumptions on the admissible class of external magnetic fields.   

{\bf Optimal Control of \eqref{CF}-\eqref{NLP}.} Numerical analysis of the optimal control of the 1D LLG equation was studied in \cite{TDMKAP} using the regular solutions of \eqref{NLP} without the conditions on the data and control. However, since the finite time blow-up of regular solutions may occur in higher dimensions even without the control  (see, \cite{BLGMCH}), the assumptions on the control and data are crucial for the 2D LLG equations. When the underlying state equation doesn't admit a (unique) strong solution for a general class of controls and data, various methods have been employed to tackle the optimal control problems associated with the model. In particular, the optimal control problem of the 3D Navier–Stokes equations was studied by different methods, for example, using the unique, strong solution obtained by a smallness condition on data and control \cite{BTK}, treating the state equation as constraint mixed by the state and control \cite{Wa}, and working with a cost functional involving the state variable belonging to $ L^8(0,{T};L^4(\Omega)),$ which is finite,  so that any weak solution becomes strong solution \cite{SSS,ECKC}. In this paper, by invoking the first method of taking the admissible class of bounded controls, we established the existence of optimal control of \eqref{CF}-\eqref{NLP}.   We derived the first-order necessary optimality condition to characterize the optimal control by the classical adjoint problem approach. This theorem requires a detailed proof of the Fr\'echet differentiability of the control-to-state operator, the solvability of the linearized system, and the adjoint system of \eqref{NLP}, which are proved with the aid of the unique regular solution of \eqref{NLP}. To the authors' knowledge, such a rigorous analysis of optimal control of the 2D LLG equation with control as the space-time-dependent external magnetic field has not been done.

The paper is organized as follows. In section 2, we have given the required function spaces and inequalities, formulated the control problem, and stated the main results. The local existence and uniqueness of a regular solution of system \eqref{NLP} are obtained in subsection \ref{SLEU}, and a global solution is proved in subsection \ref{SECGE}. Section \ref{SEOC} discusses the existence of optimal control. The existence and uniqueness of the linearized system, adjoint system, and the differentiability of the control-to-state operator are discussed in section \ref{SOC}. Finally, subsection \ref{SFOOC} is devoted to a first-order optimality condition.

\section{Function Spaces and Main Results}

\subsection{Function Spaces and Inequalities}\label{FS}

We state some of the basic cross-product properties without proof which is used throughout the paper. 
\begin{Lem}\label{CPP}
	Let $a,b$ and $c$ be three vectors of $\mathbb{R}^3$, then the following vector identities hold: $a\cdot(b \times c)=-(b \times a)\cdot c$,  $a \cdot (a \times b)=0$,  $a \times (b \times c)=(a \cdot c) b - (a \cdot b) c.$ Moreover, assume that $1 \leq r,s \leq \infty, \ (1/r)+(1/s)=1$ and $p\geq 1$, then if $f\in L^{pr}(\Omega)$ and $g\in L^{ps}(\Omega),$ we have
	$$\|f \times g\|_{L^p(\Omega)} \leq \|f\|_{L^{pr}(\Omega)} \|g\|_{L^{ps}(\Omega)}.$$
\end{Lem}
The $L^2$ theory of Laplace operator with Neumann  boundary condition leads to the following inequality of norms that will be quite useful. 
\begin{Lem}[see, \cite{KW}]\label{EN}
	Let $\Omega$ be a bounded smooth domain in $\mathbb{R}^n$ and $k \in \mathbb{N}$. There exists a constant $C_{k,n}>0$ such that for all $m \in H^{k+2}(\Omega)$ and $\frac{\partial m}{\partial \eta}\big|_{\partial \Omega}=0,$ it holds that
	\begin{eqnarray}\label{ES1}
		\|m\|_{H^{k+2}(\Omega)} \leq C_{k,n} \left(\|m\|_{L^2(\Omega)}+ \|\Delta m\|_{H^k(\Omega)}\right).	
	\end{eqnarray}
\end{Lem}
The above lemma allows us to define a norm on $H^{k+2}(\Omega)$ as follows 
$$\|m\|_{H^{k+2}(\Omega)}:=\|m\|_{L^2(\Omega)}+\|\Delta m\|_{H^k(\Omega)}.$$
The following inequalities will be frequently used in the paper.
\begin{Pro}
	Let $\Omega$ be a regular bounded subset of  $\mathbb{R}^2$. There exists a constant $C>0$ depending on $\Omega$ such that for all $m \in H^2(\Omega)$ with $\frac{\partial m}{\partial \eta}\big|_{\partial\Omega}=0,$ we have
	\begin{eqnarray} 
		\|m\|_{L^\infty(\Omega)} &\leq& C\ \left(\|m\|^2_{L^2(\Omega)}+ \|\Delta m\|^2_{L^2(\Omega)}\right)^{\frac{1}{2}},\label{ES2}\\
		\|\nabla m\|_{L^s(\Omega)} &\leq& C\  \|\Delta m\|_{L^2(\Omega)}, \ \ \forall \ s \in [1,\infty),\label{ES3}\\
		\|D^2m\|_{L^2(\Omega)} &\leq& C\ \|\Delta m\|_{L^2(\Omega)}.\label{ES4}	
	\end{eqnarray}
	Moreover, for every $m \in H^3(\Omega)$ with $\frac{\partial m}{\partial \eta}\big|_{\partial\Omega}=0,$ we have
	\begin{eqnarray}
		\|\Delta m\|_{L^2(\Omega)} &\leq& C\ \|\nabla \Delta m\|_{L^2(\Omega)},\label{ES7}\\ 	
		\|D^3m\|_{L^2(\Omega)} &\leq& C\ \|\nabla \Delta m\|_{L^2(\Omega)},\label{ES10}\\
		\|D^2m\|_{L^3(\Omega)} &\leq& C\ \|\Delta m\|^\frac{1}{2}_{L^2(\Omega)} \|\nabla \Delta m\|^\frac{1}{2}_{L^2(\Omega)}.\label{ES9} 		
	\end{eqnarray}
\end{Pro}
\begin{proof}
	The inequality \eqref{ES2} results from the estimate \eqref{ES1} and the embedding $H^2(\Omega) \hookrightarrow L^\infty(\Omega)$. By spectral decomposition 
	$$\|\nabla m\|^2_{L^2(\Omega)} = \sum_{n=0}^{\infty} (\rho_n-1)\ |\langle m, \xi_n \rangle|^2 \leq \ C\sum_{n=0}^{\infty} (\rho_n-1)^2\ |\langle m,\xi_n\rangle |^2=\ C\  \|\Delta m\|^2_{L^2(\Omega)},$$ 	
	where $\rho_n$ is the eigenvalues of the operator $-\Delta + I$ and $\xi_n$ is the corresponding orthonormal eigenfunctions in $L^2(\Omega).$ Now, consider the average map $\widehat{m}=\frac{1}{|\Omega|} \int_\Omega m\ dx$, then $\int_\Omega (m-\widehat{m}) dx=0.$ By virtue of the Poincar\'e inequality (see, Section 7.10.2, \cite{SS}), the embedding $H^1(\Omega)\hookrightarrow L^s(\Omega)$ for $s \in [1,\infty)$ and estimate \eqref{ES1}, we derive
	\begin{align*}
		\|\nabla m\|_{L^s(\Omega)} &= \|\nabla (m-\widehat{m})\|_{L^s(\Omega)} \leq C \ \|\nabla (m-\widehat{m})\|_{H^1(\Omega)}\\
		&\leq C\ \left( \|m-\widehat{m}\|_{L^2(\Omega)} + \|\Delta (m-\widehat{m})\|_{L^2(\Omega)} \right)\\
		&\leq	C\ \left( \|\nabla (m-\widehat{m})\|_{L^2(\Omega)} + \|\Delta (m-\widehat{m})\|_{L^2(\Omega)} \right)
		\leq C\ \|\Delta m\|_{L^2(\Omega)}
	\end{align*}	
	which yields the estimate \eqref{ES3}. The inequality \eqref{ES4} is a result of regularity of Laplacian operator (see, \cite{KW}). By doing an integration by parts and using estimate \eqref{ES3}, one can obtain \eqref{ES7}. Next, for estimate \eqref{ES10},  we infer from Lemma \ref{EN} that
	$$\|D^3m\|_{L^2(\Omega)} \leq C\ \left(\|m\|_{L^2(\Omega)}+\|\Delta m\|_{L^2(\Omega)}+\|\nabla \Delta m\|_{L^2(\Omega)}\right).$$
Replacing $m$ by $m-\widehat{m}$ in the above inequality, applying the Poincar\'e inequality as before and using the estimates \eqref{ES3} and \eqref{ES7}, we derive
	\begin{align*}
		\|D^3 (m-\widehat{m})\|_{L^2(\Omega)} &\leq C\ \left(\|m-\widehat{m}\|_{L^2(\Omega)}+\|\Delta (m-\widehat{m})\|_{L^2(\Omega)} + \|\nabla \Delta (m-\widehat{m})\|_{L^2(\Omega)}\right)\\
		&\leq C\ \left( \|\nabla (m-\widehat{m})\|_{L^2(\Omega)} + \|\Delta (m-\widehat{m})\|_{L^2(\Omega)}+ \|\nabla \Delta (m-\widehat{m})\|_{L^2(\Omega)} \right)\\
		&\leq	C\ \left( \|\Delta  (m-\widehat{m})\|_{L^2(\Omega)} + \|\nabla \Delta (m-\widehat{m})\|_{L^2(\Omega)} \right)
		\leq C\ \|\nabla \Delta (m-\widehat{m})\|_{L^2(\Omega)}.
	\end{align*}
	By applying the fractional embedding $H^{\frac{1}{2}}(\Omega) \hookrightarrow L^3(\Omega)$ (see \cite{NPV},  Theorem 6.7) and using the interpolation inequality (see \cite{SK}, Theorem 2.7.2), estimates \eqref{ES4} and \eqref{ES10}, we derive
\begin{align*}
	\|D^2m\|_{L^3(\Omega)} &\leq C\ \|D^2m\|_{H^{1/2}(\Omega)} \leq C\ \|D^2m\|^{1/2}_{L^2(\Omega)}\|D^2m\|^{1/2}_{H^1(\Omega)}\\
	&\leq C \left(\|D^2m\|_{L^2(\Omega)}+ \|D^2m\|^{1/2}_{L^2(\Omega)} \|D^3m\|^{1/2}_{L^2(\Omega)}\right)\\
	&\leq C\ \left(\|\Delta m\|_{L^2(\Omega)}+ \|\Delta m\|^{1/2}_{L^2(\Omega)} \|\nabla \Delta m\|^{1/2}_{L^2(\Omega)}\right)\leq C\ \|\Delta m\|^{1/2}_{L^2(\Omega)} \|\nabla \Delta m\|^{1/2}_{L^2(\Omega)}.
\end{align*}	
This completes the proof of the estimate \eqref{ES9}. Hence the proof.
\end{proof}

\begin{Pro}{(Gagliardo-Nirenberg interpolation inequality, see, \cite{LN})}	
	Let $\Omega$ be a bounded Lipschitz domain. Suppose $1 \leq p,q,r,s \leq \infty$ and $\alpha,\beta$ are non-negative integers and $\theta \in [0,1]$ are real numbers such that $\frac{\alpha}{\beta} \leq \theta \leq 1$	and
	$\frac{1}{p}=\frac{\alpha}{n}+\left(\frac{1}{q}-\frac{\beta}{n}\right)\theta + \frac{1-\theta}{r}$. Suppose furthermore that $v$ is a function in $L^r(\Omega)$ with $\beta^{th}$ weak derivative in $L^q(\Omega)$. Then there exists constants $C_1,C_2$ depending on $\Omega,\alpha,\beta,q,r,s,\theta$ such that 
	$$\|D^\alpha v\|_{L^p(\Omega)} \leq C_1\ \|D^\beta v\|^\theta_{L^q(\Omega)} \|v\|^{1-\theta}_{L^r(\Omega)} +C_2 \|v\|_{L^s(\Omega)}$$
	with the exception that if $1<q<\infty$ and $\beta-\alpha -n/q \in \mathbb{N}$, we must choose $\alpha/\beta \leq \theta <1$. 
\end{Pro}	
	
	Choose $n=q=r=2$ and $\alpha=0$. Then some particular cases of the Gagliardo-Nirenberg inequality combined with estimates \eqref{ES3}, \eqref{ES4}, \eqref{ES7} and \eqref{ES10} are given as follows:
\begin{numcases}{}
	\|\nabla m\|_{L^4(\Omega)} \leq C\ \|\nabla m\|^{\frac{1}{2}}_{L^2(\Omega)} \|\Delta m\|^{\frac{1}{2}}_{L^2(\Omega)} \ \ \ \ \ \ \ \  \text{for}\ \  p=4,\ \  \beta=1,\label{ES5}	\\
	\|\nabla m\|_{L^6(\Omega)} \leq C\ \|\nabla m\|^{\frac{1}{3}}_{L^2(\Omega)} \|\Delta m\|^{\frac{2}{3}}_{L^2(\Omega)}\ \ \ \ \ \ \ \ \text{for}\ \  p=6,\ \ \beta=1,\label{ES6}\\
	\|\nabla m\|_{L^\infty(\Omega)} \leq C\ \|\nabla m\|^{\frac{1}{2}}_{L^2(\Omega)} \|\nabla \Delta m\|^{\frac{1}{2}}_{L^2(\Omega)}\ \ \ \ \ \text{for}\ \  p=\infty,\ \beta=2.\label{ES8}	
\end{numcases}

The following inequality will be employed to estimate the nonlinear terms. 
	If $u\in H^1(\Omega)$ and $v \in H^2(\Omega)$, then $uv\in H^1(\Omega),$ that is, there exists a constant $C>0$ such that the following holds:
	\begin{eqnarray}\label{AIN}
	\|uv\|_{H^1(\Omega)} \leq C\ \|u\|_{H^1(\Omega)}\|v\|_{H^2(\Omega)}.	
	\end{eqnarray}
The proof follows from the embeddings $H^2(\Omega)\hookrightarrow L^\infty(\Omega)$ and $H^1(\Omega)\hookrightarrow L^4(\Omega)$. Indeed, by the above embeddings, we have
	\begin{eqnarray*}
		\lefteqn{	\|uv\|_{L^2(\Omega)}+\|\nabla u\ v\|_{L^2(\Omega)}+\|u\ \nabla v\|_{L^2(\Omega)} }\\
		&& \leq \|u\|_{L^2(\Omega)}\|v\|_{L^\infty(\Omega)}+ \|\nabla u\|_{L^2(\Omega)}\|v\|_{L^\infty(\Omega)}+\|u\|_{L^4(\Omega)} \|\nabla v\|_{L^4(\Omega)}
		\leq C\ \|u\|_{H^1(\Omega)}\|v\|_{H^2(\Omega)}.
	\end{eqnarray*}
We also need the following classical comparison results. 	
\begin{Pro}{(Comparison Lemma)} \label{CLO}
	Let  $f:[0,T]\times \mathbb{R} \to \mathbb{R}$ be a continuous function  in $t$ and locally Lipschitz with respect to $\zeta$ such that $\zeta:[0,T] \to  \mathbb{R}$ satisfies $\zeta'(t)\leq f(t,\zeta)$, $\zeta(0)=\zeta_0$.
	Now, if $\eta:[0,\tilde{T})\to \mathbb{R}$ be a solution of $\eta'(t)=f(t,\eta)$ with initial condition $\eta(0)=\zeta_0$ for some $\tilde{T} \leq T$.	Then, $\zeta(t) \leq \eta(t), \ \ t\in[0,\tilde{T}).$
\end{Pro}

\subsection{Main Results}\label{MR}
In order to prove the existence of regular solution to system \eqref{NLP}, inspired by the method in \cite{GCPF}, we will first show that the following equivalent problem  admits a regular solution and  show that solution satisfies $|m|=1.$  Indeed, taking  dot product of \eqref{NLP} with $m$ and applying the properties of cross product stated in Lemma \ref{FS}, we see that $m$ and $m_t$ are orthogonal in space and time, that is, we get the pointwise identity $(d/dt)|m(\cdot,t)|^2=0.$  This shows  that $|m(x,t)|^2=1,$ pointwise space and time, since $m_0 \in \mathbb{S}^2.$ Consequently, by expanding the cross product $m\times (m\times\Delta m)$ in \eqref{NLP} and for a regular solution $m,$ we can use the identity  
\begin{eqnarray}\label{VPI}
	 \Delta |m|^2=2(m\cdot \Delta m)+ 2|\nabla m|^2,
\end{eqnarray}
to  arrive at the equivalent system of \eqref{NLP}:
\begin{equation}\label{EP}
	\begin{cases}
		m_t- \Delta m= |\nabla m|^2m +  m \times \Delta m  + m \times u- m \times (m \times u), \ \ \ (x,t)\in \Omega_T,\\
		
		\frac{\partial m}{\partial \eta}=0, \ \ \ \ \ \ \ (x,t) \in  \partial\Omega_T,\\
		
		m(\cdot,0)=m_0 \ \ \text{in} \ \Omega.
	\end{cases}	
\end{equation}
Further, due to the critical nonlinearity in the LLG equation \eqref{EP}, it seems difficult to directly get the existence of global regular solution by imposing certain smallness condition on initial data and control as done in \cite{PCCF}.  So, instead of that we first prove a local in time existence and uniqueness  of regular solution and validate that $|m|=1.$ Then, by showing some norm boundedness of that local solution, we extend  such solutions to the entire time interval. 

The following are the two main theorems concerning the well-posedness of the problem \eqref{NLP}.
\begin{Thm}[Local Existence]\label{LOCAL}	
	Suppose $m_0$ satisfies \eqref{IC} and the control $u \in L^2(0,T;H^1(\Omega))$. Then there exist a time $T^*$ depending on the size of the initial data and control such that system \eqref{NLP} admits a unique regular solution $m \in L^2(0,\tilde{T};H^3(\Omega))  \cap C([0,\tilde{T}];H^2(\Omega))$ \  for every $\tilde{T} <T^*$.
\end{Thm}
\begin{Thm}[Global Existence]\label{GLOBAL} Suppose the initial data $m_0$ satisfies \eqref{IC}, and in addition 
	assume that  $m_0$ and the control $u$ satisfies $\|\nabla m_0\|^2_{L^2(\Omega)} + \|u\|^2_{L^2(0,T;L^2(\Omega))} \leq \frac{1}{16C^*}$, where $C^*$ is the constant depends only on $\Omega$. Then the system \eqref{NLP} admits a unique regular solution $m \in L^2(0,T;H^3(\Omega))  \cap C([0,T];H^2(\Omega))$ on $[0,T]$.
		
	Moreover, there exists a constant $C(\Omega,T)>0$ such that the following estimate holds:
\begin{eqnarray}\label{SSE}
\lefteqn{\sup_{t \in [0,T]} \|m(t)\|^2_{H^2(\Omega)}+\|m\|^2_{L^2(0,T;H^3(\Omega))} + \|m_t\|^2_{L^2(0,T;H^1(\Omega))}} \nonumber\\
&&\leq  \exp\left\{C\left(1+\|\Delta m_0\|^2_{L^2(\Omega)}+\|u\|^2_{L^2(0,T;H^1(\Omega))}\right)^2\right\}. 
\end{eqnarray}

\end{Thm}

\noindent Before stating the existence of optimal solutions and first-order optimality conditions, let us define some function spaces and norms that we have used throughout this paper. The existence and uniqueness of the regular solution show that the norm of the distributed control can not be weaker than that required above. 

\begin{Def}
	Let $E_{mf}$ be a prescribed positive constant. The set of all admissible controls $\mathcal{U}_{ad}$ is defined as follows: 
$$\mathcal{U}_{ad}:= \left\{ u \in L^2(0,T;H^1(\Omega))\ \big| \  \|u\|^2_{L^2(0,T;L^2(\Omega))} \leq E_{mf} \right\}. $$
\end{Def}
From Theorem \ref{GLOBAL}, it is evident that the bound for the control parameter can be written exactly as $E_{mf}=\frac{1}{16C^*} - \|\nabla m_0\|^2_{L^2(\Omega)}$. For any fixed $E_{mf}$, the set $\mathcal{U}_{ad}$ is a bounded, convex and closed subset of the space $L^2(0,T;H^1(\Omega))$. The corresponding admissible solution space:
$$\mathcal{M}:=W^{1,2}(0,T;H^3(\Omega),H^1(\Omega))=\left\{ m \in L^2(0,T;H^3(\Omega)) \ | \  m_t \in L^2(0,T;H^1(\Omega)) \right\}.$$
Define the norms $$\|m\|_{\mathcal{M}}:=\|m\|_{L^2(0,T;H^3(\Omega))} + \|m_t\|_{L^2(0,T;H^1(\Omega))}\ \ \text{and}\ \ \|(m,u)\|_{\mathcal{M}\times L^2(0,T;H^1(\Omega))}=\|m\|_{\mathcal{M}}+\|u\|_{L^2(0,T;H^1(\Omega))}.$$
A pair $(m,u) \in \mathcal{M} \times \mathcal{U}_{ad}$ is called an admissible pair if $(m,u)$ satisfies \eqref{EP} and $\mathcal J(m,u) <+\infty$. Let us denote the set of all admissible pair as $\mathcal{A}.$
Also, an admissible pair $(\widetilde{m},\widetilde{u}) \in \mathcal{A}$ is called an optimal solution if  it minimizes the cost functional  $\mathcal J({m},{u})$, that is, $\displaystyle \mathcal J(\widetilde{m},\widetilde{u}) = \inf_{(m,u) \in \mathcal{A}} \mathcal J(m,u).$ 
Thus, the optimal control problem is stated as follows:
\begin{equation*}\label{OCP}
	\text{(OCP)}\begin{cases}
		\text{minimize}\ \mathcal J(m,u)\\
		(m,u) \in \mathcal{A}.
	\end{cases}	
\end{equation*}
In what follows, we state the existence of optimal control for (OCP).

\begin{Thm}\label{EOC}
	Suppose $m_0$ and control $u$ satisfies the assumptions of Theorem \ref{GLOBAL}. Then the optimal control problem (OCP) admits a solution $(\widetilde{m},\widetilde{u})$ such that $(\widetilde{m},\widetilde{u})\in\mathcal  M\times\mathcal U_{ad}$  is a regular solution pair of system \eqref{NLP} with initial data $m_0.$ 
\end{Thm}
Finally,  we obtain the first-order necessary optimality condition given by  a variational inequality using the classical adjoint problem approach. We mainly follow the techniques used in \cite{FART},\cite{SSS},\cite{FT} to obtain the optimality condition.

We formally derive the adjoint system corresponding to the (OCP). Consider the formal Lagrangian defined by (see, \cite{FT}) 
$$\mathcal{L}(m,u,\phi)=\mathcal J(m,u)-\int_0^T (N(m,u),\phi)\ dt,$$
where  $\phi$ is the adjoint variable and $N(m,u)=	m_t- \Delta m- |\nabla m|^2m -  m \times \Delta m  - m \times u +  m \times (m \times u)$. If $(\widetilde{m},\widetilde{u})$ is an optimal solution, then $\frac{\partial \mathcal{L}}{\partial m} (\widetilde{m},\widetilde{u},\phi)\cdot m = 0$ for all $m$ satisfying $m(0)=0$. This leads to the following \textit{adjoint system}
\begin{equation}\label{AP}
	\begin{cases}
		-\phi_t -  \Delta \phi=   \Delta(\phi \times \widetilde{m})+  (\Delta \widetilde{m} \times \phi) + (\widetilde{u} \times \phi) - 2 \nabla \cdot \{(\widetilde{m}\cdot \phi)\nabla \widetilde{m}\}\\[1 mm]
		\hspace{0.9in} + |\nabla \widetilde{m}|^2 \phi +  (\phi \times \widetilde{m})\times \widetilde{u} +  \phi \times (\widetilde{m} \times \widetilde{u})+ (\widetilde{m}-m_d) \ \ \ \text{in}\ \Omega_T,\\[1 mm]
		\frac{\partial \phi}{\partial \eta}=0 \ \ \ \ \ \ \text{on} \ \partial \Omega_T,\\
		\phi(T)=\widetilde{m}(T)-m_\Omega \ \ \ \text{in} \ \ \Omega.
	\end{cases} 
\end{equation}
\begin{Thm}\label{FOOC} 	Suppose $m_0$ and control $u$ satisfies the assumptions of Theorem \ref{GLOBAL}.
	Let $\widetilde{u} \in \mathcal{U}_{ad}$ be an optimal control of (OCP) with associated state $\widetilde{m}.$ Then there exists a unique element $\phi \in W^{1,2}(0,T;H^1(\Omega),H^1(\Omega)^*)$ corresponding to the admissible pair $(\widetilde{m},\widetilde{u})$ such that the triplet $(\phi,\widetilde{m},\widetilde{u})$ satisfies the adjoint system \eqref{AP} weakly. Moreover, the following variational inequality holds:
	\begin{eqnarray*}
		\lefteqn{\int_{\Omega_T} \widetilde{u}\cdot (u-\widetilde{u})\ dx\ dt + \int_{\Omega_T} \nabla \widetilde{u}\cdot (\nabla u-\nabla \widetilde{u})\ dx\ dt}\\
		&&+\int_{\Omega_T} \Big((\phi \times \widetilde{m})+  \widetilde{m} \times (\phi \times \widetilde{m})\Big)\cdot (u-\widetilde{u})\ dx\ dt \geq 0, \ \forall \ u \ \in \mathcal{U}_{ad}.
	\end{eqnarray*}
\end{Thm}
The proof of Theorems \ref{LOCAL}-\ref{FOOC} are given in Sections 3,4 and 5. 

\section {Local and Global Solvability of the Direct Problem}
\subsection{Local Time Existence and Uniqueness of Smooth Solutions}\label{SLEU}

By the method of Galerkin approximation, we first solve the equivalent system \eqref{EP} instead of \eqref{NLP} since \eqref{EP} has the advantage of having an elliptic part to derive appropriate priori estimates. We further show that any regular solution $m \in L^2(0,T;H^3(\Omega))$ of \eqref{EP} also satisfies $|m|=1$, which will in turn solve the system \eqref{NLP}.  The proof mainly follows the strategy used in \cite{GCPF}. Since the authors (\cite{GCPF})  mainly proved the local time existence of regular solutions for 3D LLG  equation with general effective field $\mathcal E_{eff}$ and without  external magnetic field $u,$ we completely prove this theorem for 2D LLG equation with specific  $u$ in semilinear form and obtain precise estimates in terms of $u$ and data $m_0.$  
\begin{proof}[Proof of Theorem \ref{LOCAL}]\ 
The proof is divided into four steps. \\	\noindent Step 1: \textit{Galerkin approximation and priori estimates}
	
	Let $\xi_i$ be the $i^{th}$ eigenfunction corresponding to the eigenvalue $\rho_i$ of the operator $-\Delta + I$ with vanishing Neumann boundary condition, that is, $(-\Delta +I) \xi_i = \rho_i \xi_i$ with $\frac{\partial \xi_i}{\partial \eta}\big|_{\partial \Omega}=0$ such that $\left\{\xi_i\right\}^{\infty}_{i=1}$ is the orthonormal basis of $L^2(\Omega)$ and orthogonal basis of $H^1(\Omega)$ and $H^2(\Omega)$. Let $W_n:=span\{\xi_1,\xi_2,...,\xi_n\}$ be the finite dimensional subspace of $L^2(\Omega)$ and $\mathbb{P}_n:L^2(\Omega) \to W_n$ be the orthogonal projection. We consider the Galerkin system
	\begin{equation}\label{GA}
		\begin{cases}
			(m_n)_t- \Delta m_n=\mathbb{P}_n \left[ |\nabla m_n|^2m_n +  m_n \times \Delta m_n + m_n \times u- m_n \times (m_n \times u)\right],\\
			m_n(0)=\mathbb{P}_n(m_0)
		\end{cases}
	\end{equation}
	where $m_n(t)=\sum_{k=1}^{n}a_{kn}(t)\ \xi_k\in W_n$ and $\mathbb{P}_n(m_0)=\sum_{k=1}^n b_{kn} \ \xi_k$. Then system \eqref{GA} is equivalent to the following system of ordinary differential equations
	\begin{equation}\label{ODE}  
		\frac{d}{dt}a_{kn}(t)= F_k(t,a_{n},u), \ \ a_{k}(0)=b_{k}, \ \ \ \ \ k=1,2,...,n,
	\end{equation}
		where $a_{n}=(a_{1n}(t),a_{2n}(t),...,a_{nn}(t))^T$ and 
	\begin{align*}
			F_k(t,a_{n},u) =& -(\rho_k-1)\ a_{kn}(t) + \sum_{q,r,s=1}^{n} \int_\Omega \bigg[ \xi_q \ (\nabla \xi_r \cdot \nabla \xi_s) \ \xi_k \bigg]\ dx \ a_{qn}(t) \ a_{rn}(t)\ a_{sn}(t)\\
			&- \sum_{r,s=1}^{n} (\rho_s-1) \int_\Omega \xi_r \ \xi_s \ \xi_k\ dx\ \big(a_{rn}(t)\times a_{sn}(t)\big) +  \sum_{r=1}^{n} \int_\Omega \big(a_{rn}(t)\times u\big)\ \xi_r\  \xi_k\  dx\\
			&- \sum_{r,s=1}^{n} \int_\Omega \big[a_{rn}(t) \times (a_{sn}(t) \times u)\big]\ \xi_r \ \xi_s \ \xi_k\ dx.	
	\end{align*}	
	Since $C([0,T];H^1(\Omega))$ is densely embedded in $L^2(0,T;H^1(\Omega))$, assume that $u \in C([0,T];H^1(\Omega))$. Then, $F_k(t,a_n,u)$ is a continuous function of $(t,a_n)$ in $[0,T]\times \mathbb{R}^{3n}$. Therefore, by the existence theory of ODE (see, \cite{PH}), there exists a solution $a_n\in C^1([0,t_m];\mathbb{R}^{3n})$, where $t_m$ is the maximal time of existence, that is, if $t_m<T$ then $|a_n(t)|$ tends to $+\infty$ as $t \to t_m$. However, we will show by an appropriate \emph{a priori} estimates that this is not the case.

	Taking $L^2$ inner product of equation \eqref{GA} with $m_n$, using $a\cdot(a\times b)=0$, applying the continuous embedding $H^2(\Omega) \hookrightarrow  L^\infty(\Omega)$, the equality of norms \eqref{ES1} and \eqref{ES3}, we obtain
	\begin{eqnarray}
		\lefteqn{\frac{1}{2} \frac{d}{dt} \|m_n(t)\|^2_{L^2(\Omega)} + \|\nabla m_n(t)\|^2_{L^2(\Omega)} = \int_\Omega |m_n(t)|^2 \ |\nabla m_n(t)|^2 dx} \nonumber\\
		&&\leq   \ \|m_n(t)\|^2_{L^\infty(\Omega)} \ \| \nabla m_n(t)\|^2_{L^2(\Omega)}\leq C(\Omega) \Big(\|m_n(t)\|^2_{L^2(\Omega)}+\|\Delta m_n(t)\|^2_{L^2(\Omega)}\Big)^2\label{L2}.
	\end{eqnarray}
	By hitting  equation \eqref{GA} with $\Delta ^2 m_n$ and doing integration by parts to each term on the right hand side, we get
	\begin{eqnarray}
	\lefteqn{\frac{1}{2}\frac{d}{dt} \| \Delta m_n(t)\|^2_{L^2(\Omega)} + \int_\Omega |\nabla  \Delta m_n(t)|^2\ dx}\nonumber\\
	&=&- \int_\Omega  \nabla (|\nabla m_n|^2m_n) \cdot \nabla \Delta m_n\ dx - \int_\Omega  \nabla \left(m_n \times \Delta m_n\right) \cdot \nabla \Delta m_n\ dx\nonumber\\
	&&- \int_\Omega  \nabla(m_n \times u) \cdot \nabla\Delta m_n\ dx + \int_\Omega  \nabla\big(m_n \times (m_n \times u)\big) \cdot \nabla\Delta m_n\ dx := \sum_{i=1}^4 E_i.\label{SOE}	
	\end{eqnarray}
	Let us estimate the bounds for each term separately. By applying H\"older's inequality, then using estimates \eqref{ES2},\eqref{ES3} and \eqref{ES9} followed by Cauchy's inequality, we derive
	\begin{flalign*}
		E_1 &= - \int_\Omega \left[2\nabla m_n(t) D^2m_n(t)\ m_n(t) \nabla \Delta m_n(t)+ |\nabla m_n(t)|^2 \nabla m_n(t)\ \nabla \Delta m_n(t)\right]\ dx&\\
		&\leq 2\  \|m_n(t)\|_{L^\infty(\Omega)} \|D^2m_n(t)\|_{L^3(\Omega)} \|\nabla m_n(t)\|_{L^6(\Omega)} \|\nabla \Delta m_n(t)\|_{L^2(\Omega)}+\  \|\nabla m_n(t)\|^3_{L^6(\Omega)} \|\nabla \Delta m_n(t)\|_{L^2(\Omega)}\\
		&\leq C\   \left(\|m_n(t)\|^2_{L^2(\Omega)}+\|\Delta m_n(t)\|^2_{L^2(\Omega)} \right)^{\frac{1}{2}}  \|\Delta m_n(t)\|^{\frac{3}{2}}_{L^2(\Omega)} \|\nabla \Delta m_n(t)\|^{\frac{3}{2}}_{L^2(\Omega)}\\
		&\ \ \ \ \  +C\    \|\Delta m_n(t)\|^{3}_{L^2(\Omega)} \|\nabla \Delta m_n(t)\|_{L^2(\Omega)}\\
		&\leq \epsilon \ \int_{\Omega}|\nabla \Delta m_n(t)|^2dx+ C(\epsilon) \left\{ \left(\|m_n(t)\|^2_{L^2(\Omega)}+\|\Delta m_n(t)\|^2_{L^2(\Omega)} \right)^5 +  \left(\|m_n(t)\|^2_{L^2(\Omega)}+\|\Delta m_n(t)\|^2_{L^2(\Omega)} \right)^3\right\}.
	\end{flalign*}
By virtue of $\nabla \Delta m_n \cdot (m_n \times \nabla \Delta m_n)=0$, the estimates \eqref{ES3} and \eqref{ES9}, one can get
\begin{flalign*}
	E_2 &= - \int_\Omega \big(\nabla m_n(t) \times \Delta m_n(t)\big)\ \nabla \Delta m_n(t)\ dx&\\
	&\leq \ \|\nabla m_n(t)\|_{L^6(\Omega)} \|\Delta m_n(t)\|_{L^3(\Omega)} \|\nabla \Delta m_n(t)\|_{L^2(\Omega)}\\
	&\leq C   \|\Delta m_n(t)\|^{\frac{3}{2}}_{L^2(\Omega)} \|\nabla \Delta m_n(t)\|^{\frac{3}{2}}_{L^2(\Omega)}\\
	&\leq  \epsilon \  \int_{\Omega}|\nabla \Delta m_n(t)|^2dx+ C(\epsilon) \  \left(\|m_n(t)\|^2_{L^2(\Omega)}+\|\Delta m_n(t)\|^2_{L^2(\Omega)} \right)^3.
\end{flalign*}
The inequalities \eqref{ES2},\eqref{ES3} and the embedding $H^1(\Omega)\hookrightarrow L^4(\Omega)$ lead to the estimates
\begin{flalign*}
	E_3 &=  - \int_\Omega \bigg[ \Big(\nabla m_n(t) \times u(t)\big) \nabla \Delta m_n(t)+\big(m_n(t) \times \nabla u(t)\big) \nabla \Delta m_n(t)\bigg]\ dx&\\
	&\leq \  \|\nabla m_n(t)\|_{L^4(\Omega)} \|u(t)\|_{L^4(\Omega)} \|\nabla \Delta m_n(t)\|_{L^2(\Omega)} + \  \|m_n(t)\|_{L^\infty(\Omega)} \|\nabla u(t)\|_{L^2(\Omega)} \|\nabla \Delta m_n(t)\|_{L^2(\Omega)}\\
	&\leq \epsilon   \ \int_{\Omega}|\nabla \Delta m_n(t)|^2dx+ C(\epsilon) \ \|u(t)\|^2_{H^1(\Omega)} \left(\|m_n(t)\|^2_{L^2(\Omega)}+\|\Delta m_n(t)\|^2_{L^2(\Omega)} \right)
\end{flalign*}
and
\begin{flalign*}
	E_4&=  \int_\Omega \bigg[ \Big(\nabla m_n(t) \times (m_n(t) \times u(t)) + m_n(t) \times (\nabla m_n(t) \times u)+m_n(t) \times (m_n(t) \times \nabla u(t))\Big) \bigg]  \nabla \Delta m_n(t)\ dx&\\
	&\leq 2\  \|\nabla m_n(t)\|_{L^4(\Omega)} \|m_n(t)\|_{L^\infty(\Omega)} \|u(t)\|_{L^4(\Omega)} \|\nabla \Delta m_n(t)\|_{L^2(\Omega)}\\
	&\ \ \ \ \ +  \  \|m_n(t)\|^2_{L^\infty(\Omega)} \|\nabla u(t)\|_{L^2(\Omega)} \|\nabla \Delta m_n(t)\|_{L^2(\Omega)}\\
	&\leq \epsilon \ \int_{\Omega}|\nabla \Delta m_n(t)|^2dx+ C(\epsilon)\ \|u(t)\|^2_{H^1(\Omega)}  \left(\|m_n(t)\|^2_{L^2(\Omega)}+\|\Delta m_n(t)\|^2_{L^2(\Omega)} \right)^2 .
\end{flalign*}
	Substituting the estimates for $E_1,E_2,E_3$ and $E_4$ in \eqref{SOE} and choosing $\epsilon=\frac{1}{8}$, we arrive at
	\begin{eqnarray}\label{H2} 
		\frac{1}{2}\frac{d}{dt} \| \Delta m_n(t)\|^2_{L^2(\Omega)} + \frac{1}{2}\|\nabla  \Delta m_n(t)\|^2_{L^2(\Omega)}\hspace{2in} \nonumber\\ 
		\leq C(\Omega) \left(1+\|u(t)\|^2_{H^1(\Omega)}\right) \Big(1+\|m_n(t)\|^2_{L^2(\Omega)}+\|\Delta m_n(t)\|^2_{L^2(\Omega)}\Big)^5.
	\end{eqnarray}
	Combining equation \eqref{L2} and \eqref{H2}, we have
	\begin{eqnarray}\label{MI}
		\frac{d}{dt} \left(\|m_n(t)\|^2_{L^2(\Omega)}+\| \Delta m_n(t)\|^2_{L^2(\Omega)}\right)+\|\nabla m_n(t)\|^2_{L^2(\Omega)}+ \|\nabla  \Delta m_n(t)\|^2_{L^2(\Omega)}\nonumber\\
		\leq  C(\Omega)\left(1+\|u(t)\|^2_{H^1(\Omega)}\right)\Big(1+\| m_n(t)\|^2_{L^2(\Omega)}+\| \Delta m_n(t)\|^2_{L^2(\Omega)}\Big)^5.\label{CE}
	\end{eqnarray}
	Define $y(t):=1+\|m_n(t)\|^2_{L^2(\Omega)}+\| \Delta m_n(t)\|^2_{L^2(\Omega)}$  and $g(t):=C(\Omega) \left(1+\|u(t)\|^2_{H^1(\Omega)}\right).$ From \eqref{MI}, it is clear that $y(t)$ satisfies the  ODE: $\frac{dy(t)}{dt}\leq g(t)y(t)^5,$ and hence by the application of Comparison Lemma (Proposition \ref{CLO}), we have 
	$$y(t) < \frac{y(0)}{\bigg(1-4 y(0)^4\int_0^t g(\tau)\ d\tau\bigg)^{\frac{1}{4}}}, \ \text{provided} \ \int_0^t g(\tau)\  d\tau < \frac{1}{4 y(0)^4}.$$
Since $\|\mathbb{P}_nm_0\|_{H^2(\Omega)} \leq C\  \|m_0\|_{H^2(\Omega)}$(see, \cite{GCRJ}), we have
$$y(0)=1+\|m_n(0)\|^2_{L^2(\Omega)}+ \|\Delta m_n(0)\|^2_{L^2(\Omega)} \leq 1+\|\mathbb{P}_nm_0\|^2_{H^2(\Omega)}\leq 1+C\  \|m_0\|^2_{H^2(\Omega)}. $$ Now, by setting $M_0:= 1+C\  \|m_0\|^2_{H^2(\Omega)}$, we get
	$$	1+\|m_n(t)\|^2_{L^2(\Omega)}+\| \Delta m_n(t)\|^2_{L^2(\Omega)}
	< \frac{M_0}{\bigg( 1-4M_0^4 C(\Omega) \int_0^t  \left(1+\|u(\tau)\|^2_{H^1(\Omega)}\right)d\tau \bigg)^{\frac{1}{4}}}$$
	as long as $ \displaystyle C(\Omega) \int_0^t \left(1+\|u(\tau)\|^2_{H^1(\Omega)}\right)  d\tau < \frac{1}{4M_0^4}.$
	If the above inequality holds for every $t\in [0,T]$, then we can directly get a uniform bound for global time. If not, then for every control $u \in L^2(0,T;H^1(\Omega)),$ there exists a time $T^*$ such that for any time $\tilde{T}<T^*$,
	$C(\Omega) \int_0^{\tilde{T}} \left(1+\|u(t)\|^2_{H^1(\Omega)}\right)  dt \leq \frac{1}{8M_0^4}$. Therefore,
	\begin{equation}\label{UB1}
		\sup_{t \in [0,\tilde{T}]} \left[1+  \| m_n(t)\|^2_{L^2(\Omega)}+\| \Delta m_n(t)\|^2_{L^2(\Omega)}  \right] \leq \sqrt[4]{2}M_0, \ \ \forall \ n\in\mathbb N.
	\end{equation}
	Using this estimate in equation \eqref{CE}, we get
	\begin{multline}\label{UB2}
		\int_0^{\tilde{T}} \left(\|\nabla m_n(t)\|^2_{L^2(\Omega)}+ \|\nabla  \Delta m_n(t)\|^2_{L^2(\Omega)}\right)  dt \\
		\leq \left(\|m_0\|^2_{L^2(\Omega)}+\|\Delta m_0\|^2_{L^2(\Omega)}\right)
		+ C(\Omega) (\sqrt[4]{2}M_0)^5 \left(T+\|u\|^2_{L^2(0,T;H^1(\Omega))}\right) \ \ \ \forall \ n\in \mathbb{N}.
	\end{multline}

	Therefore, as a result of estimate \eqref{UB1}, \eqref{UB2} and Lemma \ref{EN}, we get that $\{m_n\}$ is uniformly bounded in $L^{\infty}(0,\tilde{T};H^2(\Omega))$ and $L^2(0,\tilde{T};H^3(\Omega))$ for every $\tilde{T}<T^*$.
	
	Next, we need to show that $(m_n)_t$ is uniformly bounded in $L^2(0,\tilde T;H^1(\Omega))$. In order to do so, consider the  $H^1$ norm of $(m_n)_t$ in \eqref{GA}. As $m_n \in H^3(\Omega),$ we have $|\nabla m_n|^2 \in H^1(\Omega)$ and hence the inequality \eqref{AIN} shows that $|\nabla m_n|^2m_n\in H^1(\Omega)$. Similarly, as $u\in H^1(\Omega)$, so both $m_n \times u$ and $m_n \times (m_n \times u)$ are in $H^1(\Omega)$. 
	
	Now, taking  the $L^2(0,\tilde{T})$ norm of $\|(m_n)_t\|_{H^1(\Omega)}$ and using the uniform bounds \eqref{UB1} and \eqref{UB2}, we can get
	$$\int_0^{\tilde{T}} \left(\|(m_n)_t(\tau)\|^2_{L^2(\Omega)}+ \|(\nabla m_n)_t(\tau)\|^2_{L^2(\Omega)}\right) \ d\tau \ \leq C \ \ \ \forall \ n.$$  
	Hence, $\{(m_n)_t\}$ is uniformly bounded in $L^2(0,\tilde{T};H^1(\Omega))$ for every $\tilde{T}<T^*$.
	
	Using the Aloglu weak$^*$ compactness and reflexive weak compactness theorems (Theorem 4.18, \cite{JCR}), we get	
	\begin{eqnarray} \left\{\begin{array}{cccll}
			m_n &\overset{w}{\rightharpoonup} & m \  &\mbox{weakly in}& \ L^2(0,\tilde{T};H^3(\Omega)),\\
			m_n &\overset{w^\ast}{\rightharpoonup} & m \ &\mbox{weak$^*$ in}& \ L^{\infty}(0,\tilde{T};H^2(\Omega)),\\
			(m_n)_t &\overset{w}{\rightharpoonup} & m_t \ &\mbox{weakly in}& \ L^2(0,\tilde{T};H^1(\Omega)), \ \ \mbox{as} \  \ n\to \infty. \label{wc}
		\end{array}\right.	
	\end{eqnarray}
	
	By using the Aubin-Lions-Simon lemma (see, Corollary 4, \cite{JS}), we can obtain a sub-sequence of $\{m_n\}$ (again denoted as $\{m_n\}$) such that $m_n \to m$ strongly in $L^2(0,\tilde{T};H^2(\Omega))$ and $C([0,\tilde T];H^1(\Omega))$.
	
	\noindent Step 2: \textit{Passing to the limit}
	
	Using \eqref{wc} and the above strong convergence results, we can show the weak convergence of each term of \eqref{GA} in $L^2(0,\tilde{T};L^2(\Omega))$. We shall prove this for the nonlinear terms on the right-hand side of \eqref{GA} with the test functions $v\in L^2(0,T;W_n)$. 
	\begin{Lem}\label{lem1}
		Suppose $\{m_n\}$ satisfies the weak convergence in \eqref{wc}. Further $m_n \to m$ strongly in $L^2(0,\tilde{T};H^2(\Omega))\cap C([0,\tilde{T}];H^1(\Omega))$. Then for any test function $v \in L^2(0,\tilde{T};W_n )$, the following convergences hold:
		\begin{enumerate}[label=(\roman*)]
			\item $\displaystyle \int_0^{\tilde{T}} \big(|\nabla m_n|^2 m_n,v \big)\ dt \to \int_0^{\tilde{T}}  \big(|\nabla m|^2 m,v\big)\ dt,$
			\item $\displaystyle \int_0^{\tilde{T}} \big(m_n\times  \Delta m_n,v\big)\ dt \to \int_0^{\tilde{T}} \big(m\times \Delta m,v\big)\ dt,$
			\item $\displaystyle \int_0^{\tilde{T}} \big(m_n\times  u,v\big)\ dt \to \int_0^{\tilde{T}} \big(m\times u,v\big)\ dt,$
			\item $\displaystyle \int_0^{\tilde{T}} \big(m_n\times (m_n \times u),v\big)\ dt \to \int_0^{\tilde{T}} \big(m\times (m \times u),v\big)\ dt$\ \ \ \ \  as $n \to \infty$.
		\end{enumerate}
	\end{Lem}
	The proof of Lemma \ref{lem1} is given after the completion of this theorem.\\
	
	\noindent As a consequence of \eqref{wc} and Lemma \ref{lem1}, taking limit in \eqref{GA} and using the denseness property of $L^2(0,\tilde{T};W_n)$ in $L^2(0,\tilde{T};L^2(\Omega))$, we conclude that $m$ satisfies the first equation of \eqref{EP} in $L^2(0,\tilde{T};L^2(\Omega))$.

	\noindent Step 3: \textit{Verification of initial data and $|m(x,t)|=1$}\\
	 From the convergence \eqref{wc}, it is evident that $m\in L^2(0,\tilde T;H^3(\Omega))$ and $m_t\in L^2(0,\tilde T;H^1(\Omega)),$ whence $m\in C([0,\tilde T];H^2(\Omega))$ (see, \cite{JCR}, Corollary 7.3). Consequently, the proof of $m(0)=m_0$ follows through the standard argument from \eqref{EP}, \eqref{GA} and Lemma \ref{lem1}.   
	 
	 For the regular solution $m\in L^2(0,\tilde{T};H^3(\Omega)) \cap L^\infty(0,\tilde{T};H^2(\Omega))$ of system \eqref{EP}, we can show that $|m(x,t)|=1$ in $\Omega_{\tilde{T}}.$ Let us take the scalar product of \eqref{EP} with $m$ and use the identity \eqref{VPI}. By the property $a \cdot(a \times b)=0$, the last three terms of \eqref{EP} are zero, and so we get
	 $$\frac{d}{dt}\big(|m|^2-1\big)-  \Delta \big(|m|^2-1\big) -2\ |\nabla m|^2 \big(|m|^2-1\big)=0.$$ By setting $y=|m|^2-1,$ the rest of the proof can be completed by deriving energy estimate for this equation (see, Theorem 1.1, \cite{GCPF}).

\noindent Step 4: \textit{Uniqueness of strong solution}
	
	Let $m_1,m_2 \in L^2(0,\tilde{T};H^3(\Omega))\cap C([0,\tilde{T}];H^2(\Omega))$ be the two regular solutions of system \eqref{EP} corresponding to $u \in L^2(0,\tilde{T};H^1(\Omega))$ and $m_0\in H^2(\Omega)$ satisfying $|m_1|=|m_2|=1.$ Then $v=m_1-m_2$ will solve the equation
	$$
	\begin{cases}
		v_t- \Delta v=((\nabla m_1+\nabla m_2)\cdot \nabla v)\  m_1+  |\nabla m_2|^2\ v+ m_1 \times \Delta v\\
		\hspace{.61in}+ v \times \Delta m_2 +  v \times u - v \times (m_1 \times u) -  m_2 \times (v \times u),\ \ (x,t) \in \Omega_{\tilde{T}},\\
		\frac{\partial v}{\partial \eta}=0 \ \ \text{on}\ \partial \Omega_{\tilde{T}},\\
		v(0)=0 \ \ \text{in}\ \Omega.	
	\end{cases}
	$$	
	By taking, inner product with $v$ and applying properties of cross product, we get
	\begin{align*}
		\lefteqn{\frac{1}{2}\frac{d}{dt}\|v(t)\|^2_{L^2(\Omega)} +  \|\nabla v(t)\|^2_{L^2(\Omega)}= \int_\Omega ((\nabla m_1+ \nabla m_2)\cdot \nabla v)\ m_1 \cdot v\ dx }\\
		& +  \int_\Omega |\nabla m_2|^2\ |v|^2\ dx +  \int_\Omega (m_1 \times \Delta v)\cdot v\ dx - \int_\Omega (m_2 \times (v \times u))\cdot v\ dx.
	\end{align*}
By integration by parts for the third term and using the embedding $H^1(\Omega) \hookrightarrow L^4(\Omega)$ for the forth term on the right-hand side, we derive
\begin{align*}
		\lefteqn{\frac{1}{2}\frac{d}{dt}\|v(t)\|^2_{L^2(\Omega)} +  \|\nabla v(t)\|^2_{L^2(\Omega)}}\\
		&\leq  \  \Big( 2\|\nabla m_1(t)\|_{L^\infty(\Omega)} + \|\nabla m_2(t)\|_{L^\infty(\Omega)} \Big)\  \|v(t)\|_{L^2(\Omega)} \|\nabla v(t)\|_{L^2(\Omega)} +  \  \|\nabla m_2(t)\|^2_{L^\infty(\Omega)} \|v(t)\|^2_{L^2(\Omega)}\\
		& \hspace{.3in}+   C\  \|u(t)\|_{H^1(\Omega)} \|v(t)\|_{L^2(\Omega)} \left(\|v(t)\|_{L^2(\Omega)} + \|\nabla v(t)\|_{L^2(\Omega)}\right).
\end{align*} 
Applying Cauchy's inequality and the embedding $H^2(\Omega) \hookrightarrow L^\infty(\Omega)$, we get
	$$\frac{d}{dt}\|v(t)\|^2_{L^2(\Omega)} +  \|\nabla v(t)\|^2_{L^2(\Omega)}\leq C \left(1+\|m_1(t)\|^2_{H^3(\Omega)}+\|m_2(t)\|^2_{H^3(\Omega)}+\|u(t)\|^2_{H^1(\Omega)}\right)\ \|v(t)\|^2_{L^2(\Omega)}.$$	
	Since $m_1, m_2 \in L^2(0,\tilde{T};H^3(\Omega)) \cap L^\infty(0,\tilde{T};H^2(\Omega))$ and $u \in L^2(0,\tilde{T};H^1(\Omega))$, using Gronwall's inequality, we arrive $\|v(t)\|_{L^2(\Omega)}=0$, $\forall\ t\in [0,\tilde{T}]$.	Therefore, the solution is unique.  
\end{proof}

\begin{proof}[Proof of Lemma \ref{lem1}]	
	Let $v \in L^2(0,\tilde{T};W_n)$ be any test function.\\
	(i) Using H\"older's inequality and continuous embeddings $H^2(\Omega) \hookrightarrow L^\infty(\Omega), H^1(\Omega) \hookrightarrow L^p(\Omega)$ for $p \in [1,\infty)$, we get
\begin{align*}
	\lefteqn{\int_0^{\tilde{T}} \left[\left(|\nabla m_n|^2 m_n,v \right) - \left(|\nabla m|^2 m,v \right)\right] dt}\\
	&= \int_0^{\tilde{T}}\left[\left(|\nabla m_n|^2 (m_n-m)+(|\nabla m_n|^2-| \nabla m|^2)m,v\right)\right]dt\\
	&\leq \int_0^{\tilde{T}} \left\| \ |\nabla m_n(t)|^2 \right\|_{L^4(\Omega)} \|m_n(t)-m(t)\|_{L^4(\Omega)} \|v(t)\|_{L^2(\Omega)} \ dt \\
	&\ \ \ \ +\ \int_0^{\tilde{T}} \|m(t)\|_{L^\infty(\Omega)} \|\nabla m_n(t) + \nabla m(t) \|_{L^4(\Omega)}\|\nabla m_n(t)-\nabla m(t)\|_{L^4(\Omega)} \|v(t)\|_{L^2(\Omega)}\  dt \nonumber\\
	&\leq C\ \|m_n\|_{L^\infty(0,\tilde{T};H^2(\Omega))}^2 \|m_n-m\|_{L^2(0,\tilde{T};H^1(\Omega))} \|v\|_{L^2(0,\tilde{T};L^2(\Omega))} \\
	&\ \ \ \ +\ C\ \|m\|_{L^\infty(0,\tilde{T};H^2(\Omega))} \|m_n+m\|_{L^\infty(0,\tilde{T};H^2(\Omega))} \|m_n-m\|_{L^2(0,\tilde{T};H^2(\Omega))} \|v\|_{L^2(0,\tilde{T};L^2(\Omega))} \to 0 \ \text{as}\  n \to \infty,
\end{align*}
	since  $m_n$ is uniformly bounded in $ L^\infty(0,\tilde{T};H^2(\Omega))$ and $m_n\to m$ strongly  in $L^2(0,\tilde{T};H^2(\Omega)$.\\ \\
	(ii) By proceeding in a simillar way for the term $m_n \times \Delta m_n$, we have
	\begin{align*}
		\lefteqn{\int_0^{\tilde{T}} \Big[\big(m_n\times  \Delta m_n,v\big)-\big(m\times \Delta m,v\big)\Big]\ dt}\\ 
		&\leq C \ \|m_n-m\|_{L^\infty(0,\tilde{T};H^1(\Omega))} \|\Delta m_n\|_{L^2(0,\tilde{T};H^1(\Omega))} \|v\|_{L^2(0,\tilde{T};L^2(\Omega))}\\
		&\ \ \ \  +C \ \|m\|_{L^\infty(0,\tilde{T};H^2(\Omega))} \|m_n-m\|_{L^2(0,\tilde{T};H^2(\Omega))} \|v\|_{L^2(0,\tilde{T};L^2(\Omega))} \to 0 \ \text{as}\ n \to \infty.
	\end{align*}
	(iii) Applying H\"older's inequality followed by the embedding $H^1(\Omega) \hookrightarrow L^4(\Omega)$, we obtain 
	\begin{align*}
		\lefteqn{\int_0^{\tilde{T}} \Big[\big(m_n\times  u,v\big)-\big(m\times u,v\big)\Big]\ dt }\\
		&\leq \int_0^{\tilde{T}}  \|m_n(t)-m(t)\|_{L^4 (\Omega)} \|u(t)\|_{L^4 (\Omega)} \|v(t)\|_{L^2 (\Omega)} dt\\  
		&\leq C \ \|m_n(t)-m(t)\|_{L^\infty(0,\tilde{T};H^1(\Omega))} \|u(t)\|_{L^2(0,\tilde{T};H^1(\Omega))} \|v(t)\|_{L^2(0,\tilde{T};L^2(\Omega))} \to 0 \ \text{as}\ n \to \infty,
	\end{align*}
	since $m_n\to m$ strongly  in $L^\infty(0,\tilde{T};H^1(\Omega))$ and $u\in L^2(0,\tilde{T};H^1(\Omega))$\\ \\
	(iv) For the last term $m_n \times (m_n \times u)$, we do the computation similar to the preceding estimate to get
	\begin{align*}
	\lefteqn{\int_0^{\tilde{T}} \Big(\big(m_n\times (m_n \times u),v\big)-\big(m\times (m \times u),v\big)\Big)\  dt}\\
	&\leq C \int_0^{\tilde{T}} \|m_n(t)\|_{L^\infty (\Omega)} \|m_n(t)-m(t)\|_{L^4 (\Omega)} \|u(t)\|_{L^4 (\Omega)} \|v(t)\|_{L^2 (\Omega)}\ dt\\
	&\hspace{.3in}+ C\int_0^{\tilde{T}}  \|m_n(t)-m(t)\|_{L^4 (\Omega)} \|m(t)\|_{L^\infty (\Omega)} \|u(t)\|_{L^4 (\Omega)} \|v(t)\|_{L^2 (\Omega)}\ dt\\
	&\leq C \ \|m_n\|_{L^\infty(0,\tilde{T};H^2(\Omega))} \|m_n-m\|_{L^\infty(0,\tilde{T};H^1(\Omega))} \|u\|_{L^2(0,\tilde{T};H^1(\Omega))} \|v\|_{L^2(0,\tilde{T};L^2(\Omega))}\\
	&\hspace{.3in}+\  C \ \|m_n-m\|_{L^\infty(0,\tilde{T};H^1(\Omega))} \|m\|_{L^\infty(0,\tilde{T};H^2(\Omega))} \|u\|_{L^2(0,\tilde{T};H^1(\Omega))} \|v\|_{L^2(0,\tilde{T};L^2(\Omega))} \to 0 \ \text{as}\  n \to \infty.
	\end{align*} 
This completes the proof of Lemma \ref{lem1}.
\end{proof}

\subsection{Global Existence of Smooth Solutions}\label{SECGE}

This section shows that the local solution established previously can be proven to be global provided the initial data and control satisfy certain smallness conditions. This result is crucial to discuss the solvability of optimal control problems.  The proof follows the strategy used in \cite{GCPF,PCCF}.

\begin{proof}[Proof of Theorem \ref{GLOBAL}]
	From Theorem \ref{LOCAL} there exists a regular solution $m \in L^2(0,\tilde{T};H^3(\Omega)) \cap C([0,\tilde{T}];H^2(\Omega))$ of system \eqref{NLP} in the interval $[0,\tilde{T}]$. Suppose $T^*<T$ is the maximal time up to which the regular solution exist.

	Then we take $L^2$ inner product of \eqref{NLP} with $-\Delta m$. By invoking the vector product property $a\cdot(b\times c)=-(b\times  a)\cdot c$, we get 
	$\big(m\times(m\times \Delta m),\Delta m\big)= -\big(m\times \Delta m,m\times \Delta m\big)$. Then applying H\"older's inequality, Cauchy's inequality and using the fact that $|m(x,t)|=1$ for all $(x,t) \in \Omega_{\tilde{T}}$, we obtain
	\begin{align*}
		\lefteqn{\frac{1}{2}\frac{d}{dt}\|\nabla m(t)\|^2_{L^2(\Omega)} + \ \int_\Omega |m(t) \times \Delta m(t)|^2\ dx }\\
		&= \big(u(t),m(t) \times \Delta m(t)\big)- \big(m(t) \times u(t),m(t) \times \Delta m(t)\big)\\
		&\leq  \|u(t)\|_{L^2(\Omega)} \|m(t) \times \Delta m(t)\|_{L^2(\Omega)}+ \|m(t) \times u(t)\|_{L^2(\Omega)} \|m(t) \times \Delta m(t)\|_{L^2(\Omega)}\\
		&\leq \frac{1}{2}\|m(t) \times \Delta m(t)\|^2_{L^2(\Omega)}+ 2\  \|u(t)\|^2_{L^2(\Omega)}.
	\end{align*}
	Therefore, integrating over $0$ to $t$, we have
	\begin{equation}\label{E1}		
		\|\nabla m(t)\|^2_{L^2(\Omega)}+\int_0^t \|m(\tau)\times \Delta m(\tau)\|^2_{L^2(\Omega)}d\tau  \leq  4\ \left(\|\nabla m_0\|^2_{L^2(\Omega)} + \|u\|^2_{L^2(0,T;L^2(\Omega))}\right).
	\end{equation}
	Recall that any regular solution of \eqref{NLP} on $[0,\tilde T]\times\Omega$ will also be a regular solution of the  equivalent problem \eqref{EP} on $[0,\tilde T]\times\Omega$. Now, taking inner product of \eqref{EP} with $-\Delta m$ in $L^2(\Omega)$, using \eqref{ES5} and the fact that $m \cdot \Delta m=-|\nabla m|^2$, we get
	\begin{align*}
		\lefteqn{\frac{1}{2} \frac{d}{dt} \|\nabla m(t)\|^2_{L^2(\Omega)} + \int_\Omega |\Delta m(t)|^2\ dx }\\
		&=- \int_\Omega | \nabla m|^2 m\cdot \Delta m\ dx -  \int_\Omega (m \times u)\cdot \Delta m\ dx +  \int_\Omega (m \times (m \times u))\cdot \Delta m\ dx\\
		&\leq \ \|\nabla m(t)\|^4_{L^4(\Omega)} + \ \|u(t)\|_{L^2(\Omega)} \|\Delta m(t)\|_{L^2(\Omega)} +\ \|u(t)\|_{L^2(\Omega)} \|\Delta m(t)\|_{L^2(\Omega)}\\
		&\leq C^* \ \|\nabla m(t)\|^2_{L^2(\Omega)} \|\Delta m(t)\|^2_{L^2(\Omega)} + \frac{1}{2} \|\Delta m(t)\|^2_{L^2(\Omega)} + 2\ \|u(t)\|^2_{L^2(\Omega)},	
	\end{align*}
where $C^*=C^4$, the constant $C$ coming from estimate \eqref{ES5}. Using the bound for $\|\nabla m(t)\|^2_{L^2(\Omega)}$  from equation \eqref{E1},
\begin{equation}\label{MCOC}
	\frac{d}{dt} \|\nabla m(t)\|^2_{L^2(\Omega)} +  \bigg[1- 8C^* \left(\|\nabla m_0\|^2_{L^2(\Omega)} + \|u\|^2_{L^2(0,T;L^2(\Omega))}\right)\bigg]\ \|\Delta m(t)\|^2_{L^2(\Omega)} \leq 4\ \|u(t)\|^2_{L^2(\Omega)}.
\end{equation}
	By assumption, the initial data and control satisfies  $  \|\nabla m_0\|^2_{L^2(\Omega)} + \|u\|^2_{L^2(0,T;L^2(\Omega))}  \leq \frac{1}{16C^*}$, so that
	$$1-8C^*\left(\|\nabla m_0\|^2_{L^2(\Omega)} + \|u\|^2_{L^2(0,T;L^2(\Omega))}\right) \geq \frac{1}{2}.$$ 
	Consequently, we have 
	$$ \frac{d}{dt} \|\nabla m(t)\|^2_{L^2(\Omega)} + \frac{1}{2}\ \|\Delta m(t)\|^2_{L^2(\Omega)} \ \leq 4 \ \|u(t)\|^2_{L^2(\Omega)},\ \ \ \forall\ t \in [0,\tilde{T}],$$
	whence
	\begin{equation}\label{E2}
		\|\nabla m(t)\|^2_{L^2(\Omega)} +\frac{1}{2} \int_0^t \|\Delta m(\tau)\|^2_{L^2(\Omega)} d\tau \leq  4\ \left(\|\nabla m_0\|^2_{L^2(\Omega)} + \|u\|^2_{L^2(0,T;L^2(\Omega))}\right), \ \ \ \forall\ t \in [0,\tilde{T}].	
	\end{equation}
	In order to obtain the regularity of solution in $L^2(0,T;H^3(\Omega))$, we again appeal to Galerkin approximated system \eqref{GA}.
	By hitting equation \eqref{GA} with $\Delta^2m_n$ and integrating by parts, we get
	\begin{eqnarray}\label{E3}
	\lefteqn{\frac{1}{2}\frac{d}{dt} \| \Delta m_n(t)\|^2_{L^2(\Omega)} + \int_{\Omega} |\nabla  \Delta m_n(t)|^2\ dx }\nonumber\\
	&=&- \int_\Omega  \nabla \left(|\nabla m_n(t)|^2m_n(t)\right) \cdot \nabla \Delta m_n(t)
	- \int_\Omega  \nabla \big(m_n(t) \times \Delta m_n(t)\big) \cdot \nabla \Delta m_n(t)\ dx\\
	&&- \int_\Omega  \nabla\big(m_n(t) \times u(t)\big) \cdot \nabla\Delta m_n(t)\ dx
	+ \int_\Omega  \nabla\big(m_n(t) \times (m_n(t) \times u(t))\big) \cdot \nabla\Delta m_n(t)\ dx 
	:= \sum_{i=1}^4 S_i. \nonumber
\end{eqnarray}
	Note that from here onwards the value of the generic constant does not depend on the approximate solution deduced within the Galerkin scheme. By using the estimate \eqref{ES4}, Gagliardo-Nirenberg inequalities \eqref{ES6} and \eqref{ES8}, we get 
	\begin{flalign*}
		S_1 &= - \int_\Omega \left[ 2  \nabla m_n (D^2m_n)m_n \cdot \nabla \Delta m_n \  - |\nabla m_n|^2 \nabla m_n \cdot \nabla \Delta m_n\right] \ dx&\\
		&\leq 2 \ \|m_n(t)\|_{L^\infty(\Omega)} \|\nabla m_n(t)\|_{L^\infty(\Omega)} \|D^2 m_n(t)\|_{L^2(\Omega)} \| \nabla \Delta m_n(t)\|_{L^2(\Omega)} +    \| \nabla m_n(t)\|^3_{L^6(\Omega)} \| \nabla \Delta m_n(t)\|_{L^2(\Omega)}\\ 
		&\leq C\ \|m_n(t)\|_{L^\infty(\Omega)} \|\nabla m_n(t)\|^{\frac{1}{2}}_{L^2(\Omega)} \|\Delta m_n(t)\|_{L^2(\Omega)} \|\nabla \Delta m_n(t)\|^{\frac{3}{2}}_{L^2(\Omega)}\\
		&\ \ \ \ + C\ \|\nabla m_n(t)\|_{L^2(\Omega)} \|\Delta m_n(t)\|^2_{L^2(\Omega)} \|\nabla \Delta m_n(t)\|_{L^2(\Omega)}\\   	
		&\leq  \epsilon    \int_{\Omega}|\nabla \Delta m_n(t)|^2dx + C(\epsilon) \ \|m_n(t)\|^4_{L^\infty(\Omega)} \| \nabla m_n(t)\|_{L^2(\Omega)}^2\ \|\Delta m_n(t)\|^4_{L^2(\Omega)}\\
		& \ \ \ \  +C(\epsilon) \  \| \nabla m_n(t)\|^2_{L^2(\Omega)}\ \|\Delta m_n(t)\|^4_{L^2(\Omega)}.
	\end{flalign*}
	For the second term $S_2$,  using the fact that $(m_n \times \nabla\Delta m_n)\cdot \nabla \Delta m_n=0$ and estimate \eqref{ES8}, we derive 
\begin{flalign*}
	S_2 &= - \int_\Omega (\nabla m_n \times \Delta m_n) \cdot\nabla \Delta m_n \ dx \leq   \|\nabla m_n(t)\|_{L^\infty(\Omega)} \|\Delta m_n(t)\|_{L^2(\Omega)} \|\nabla \Delta m_n(t)\|_{L^2(\Omega)}&\\
	&\leq \epsilon \int_{\Omega}|\nabla \Delta m_n(t)|^2dx + C(\epsilon) \ 
	\|\nabla m_n(t)\|^2_{L^2(\Omega)} \|\Delta m_n(t)\|^4_{L^2(\Omega)}.
\end{flalign*}
	For the integrals $S_3$ and $S_4$, applying H\"{o}lder's inequality, the embeddings $H^1(\Omega) \hookrightarrow L^4(\Omega)$ and \eqref{ES3}, we obtain
	\begin{flalign*}
	S_3 &= - \int_\Omega (\nabla m_n \times u)\cdot \nabla \Delta m_n \ dx- \int_\Omega (m_n \times \nabla u)\cdot\nabla \Delta m_n \ dx&\\
	& \leq  \ \|\nabla m_n(t)\|_{L^4(\Omega)} \|u(t)\|_{L^4(\Omega)} \|\nabla \Delta m_n(t)\|_{L^2(\Omega)} + \ \|m_n(t)\|_{L^\infty(\Omega)} \|\nabla u(t)\|_{L^2(\Omega)} \|\nabla \Delta m_n(t)\|_{L^2(\Omega)}\\
	&\leq  \epsilon \int_{\Omega}|\nabla \Delta m_n(t) |^2dx + C(\epsilon) \ 
	 \|\Delta m_n(t)\|^2_{L^2(\Omega)} \|u(t)\|^2_{H^1(\Omega)}
	 + C(\Omega)\ \|m_n(t)\|^2_{L^\infty(\Omega)} \|u(t)\|^2_{H^1(\Omega)}.
	\end{flalign*}
	and
	\begin{flalign*}
	S_4 &= \int_\Omega (\nabla m_n \times (m_n \times u))\cdot \nabla \Delta m_n \ dx+\int_\Omega ( m_n \times (\nabla m_n \times u))\cdot \nabla \Delta m_n \ dx&\\
		&\ \ \ \ \ \ \ \ \ \ \ + \int_\Omega (m_n \times (m_n \times \nabla u))\cdot \nabla \Delta m_n \ dx\\
		& \leq  2\ \|m_n(t)\|_{L^\infty(\Omega)} \|\nabla m_n(t)\|_{L^4(\Omega)} \|u(t)\|_{L^4(\Omega)} \|\nabla \Delta m_n(t)\|_{L^2(\Omega)} +\  \|m_n(t)\|^2_{L^\infty(\Omega)} \|\nabla u(t)\|_{L^2(\Omega)} \|\nabla \Delta m_n(t)\|_{L^2(\Omega)}\\
		&\leq  \epsilon\ \int_{\Omega}|\nabla \Delta m_n(t) |^2dx + C(\epsilon)\ \|m_n(t)\|^2_{L^\infty(\Omega)} \left(\| m_n(t)\|^2_{L^\infty(\Omega)}+\|\Delta m_n(t)\|^2_{L^2(\Omega)}\right)\ \|u(t)\|^2_{H^1(\Omega)}.
	\end{flalign*}
	Choosing $\epsilon=\frac{1}{8}$ and substituting estimates for $S_1,S_2,S_3$ and $S_4$ in equation \eqref{E3}, we have
	\begin{eqnarray*}
	\lefteqn{\frac{1}{2} \frac{d}{dt} \|\Delta m_n(t)\|^2_{L^2(\Omega)}+ \frac{1}{2} \int_\Omega | \nabla \Delta m_n(t)|^2\ dx}\\ 
	&&\leq C\Big[\| \nabla m_n(t)\|_{L^2(\Omega)}^2 \left(1+\|m_n(t)\|^4_{L^\infty(\Omega)}\right) \|\Delta m_n(t)\|^4_{L^2(\Omega)} \\ 
	&&+ \left(\|m_n(t)\|^2_{L^\infty(\Omega)} + \|\Delta m_n(t)\|^2_{L^2(\Omega)}\right) \left(1+\|m_n(t)\|^2_{L^\infty(\Omega)} \right) \|u(t)\|^2_{H^1(\Omega)} \Big],
	\end{eqnarray*}
	where C depends on $\Omega$.
	

	By integrating over $0$ to $t$ and recalling from Theorem \ref{LOCAL} that $\{m_n\}$ is uniformly bounded in $L^\infty (0,\tilde{T};H^2(\Omega)) \cap  L^2(0,\tilde{T};H^3(\Omega))$, we see that the left-hand side integrals are sequentially lower semi-continuous. Besides, using the strong convergence of $\{m_n\}$ in $C([0,\tilde{T}];H^1(\Omega)) \cap L^2(0,\tilde{T};H^2(\Omega)),$ \eqref{E2} and the fact that $|m(x,t)|=1, \ \forall  (x,t)\in \Omega_{\tilde{T}}$, we conclude that 
	\begin{multline}\label{E4}
		\|\Delta m(t)\|^2_{L^2(\Omega)} +  \int_0^t \| \nabla \Delta m(\tau)\|^2_{L^2(\Omega)}\ d\tau \leq \|\Delta m_0\|^2_{L^2(\Omega)}  \\
		+C\ \bigg[  \left(\|\nabla m_0\|^2_{L^2(\Omega)}+\|u\|^2_{L^2(0,T;L^2(\Omega))}\right) \int_0^t  \ \left( \|\Delta m(\tau)\|^2_{L^2(\Omega)}\right)^2 \ d\tau\\
		+  \int_0^t  \ \left(1 + \|\Delta m(\tau)\|^2_{L^2(\Omega)}\right) \  \|u(\tau)\|^2_{H^1(\Omega)} \ d\tau \bigg],\ \ \forall\ t\in [0,\tilde{T}].
	\end{multline}	
	Applying Gronwall's inequality and using the estimate \eqref{E2}, we derive
	$$\|\Delta m(t)\|^2_{L^2(\Omega)}	\leq \exp\left\{C\left(1+\|\Delta m_0\|^2_{L^2(\Omega)}+\|u\|^2_{L^2(0,T;H^1(\Omega))}\right)^2\right\}\  \ \ \forall \ t \leq \tilde{T} <T^*.$$
	Since the right-hand side is bounded for every $0<\tilde{T}<T^*,$   $\|\Delta m(t)\|^2_{L^2(\Omega)}$ is uniformly bounded for all $t\in [0,T^*].$ Therefore, from \eqref{E4}, the regular solution exist on the entire interval $[0,T]$.\\ \\
	Now, we will derive the main energy estimate of this theorem. From the previous estimates, we know that
	\begin{equation}\label{E5}
	\sup_{t \in [0,T]} \|\Delta m(t)\|^2_{L^2(\Omega)} \leq \exp\left\{C\left(1+\|\Delta m_0\|^2_{L^2(\Omega)}+\|u\|^2_{L^2(0,T;H^1(\Omega))}\right)^2\right\}.
	\end{equation}
	Then applying this bound in estimate \eqref{E4}, we get
	\begin{equation}\label{E6}
	 \int_0^T \| \nabla \Delta m(\tau)\|^2_{L^2(\Omega)}\ d\tau \leq  \exp\left\{C\left(1+\|\Delta m_0\|^2_{L^2(\Omega)}+\|u\|^2_{L^2(0,T;H^1(\Omega))}\right)^2\right\}.		
	\end{equation}
	Next,  take $H^1$-norm of $m_t$ in \eqref{NLP} and applying \eqref{AIN}, we notice that
	\begin{eqnarray*}
	\|m\times(m\times\Delta m)\|_{H^1(\Omega)}&\leq& C\|m\|_{H^2(\Omega)}\|m\times\Delta m\|_{H^1(\Omega)} \leq C\|m\|^2_{H^2(\Omega)}\|\Delta m\|_{H^1(\Omega)} \\
		\|m\times(m\times u)\|_{H^1(\Omega)}&\leq& C\|m\|_{H^2(\Omega)}\|m\times u\|_{H^1(\Omega)} \leq C\|m\|^2_{H^2(\Omega)}\|u\|_{H^1(\Omega)} .
	\end{eqnarray*}
By doing similar estimates for the remaining terms of \eqref{NLP} and	the integration of  $\|m_t(t)\|^2_{H^1(\Omega)}$ over $(0,T)$ lead to the estimate
	\begin{eqnarray*}
		\lefteqn{\int_0^T \Big(\|m_t(t)\|^2_{L^2(\Omega)}+\|\nabla m_t(t)\|^2_{L^2(\Omega)}\Big) \ dt}\\ 
		&&\leq C \int_0^T \left( \|m(\tau)\|^4_{H^2(\Omega)} \|m(\tau)\|^2_{H^3(\Omega)} + \|m(\tau)\|^4_{H^2(\Omega)} \|u(\tau)\|^2_{H^1(\Omega)}\right) d\tau 
	\end{eqnarray*}
	where $C>0$ depends on $\Omega$. Now, using the bounds for $m$ in $L^\infty(0,T;H^2(\Omega))$ and $L^2(0,T;H^3(\Omega))$ from estimates \eqref{E5} and \eqref{E6} respectively, we obtain
	 $$\|m_t\|^2_{L^2(0,T;H^1(\Omega))}\leq  C(\Omega,T)\ \exp\left\{C\left(1+\|\Delta m_0\|^2_{L^2(\Omega)}+\|u\|^2_{L^2(0,T;H^1(\Omega))}\right)^2\right\}.$$
	By combining this inequality with estimates \eqref{E5} and \eqref{E6}, we obtain the  estimate \eqref{SSE}.
\end{proof}	
\section{Existence of optimal control}\label{SEOC}

In the previous section, we proved that if the initial data $m_0\in H^2(\Omega)$ satisfies \eqref{IC} and the control $u\in L^2(0,T;H^1(\Omega))$ are sufficiently small, then there exists a unique regular solution $m \in \mathcal{M}$ (Theorem \ref{GLOBAL}) to the system \eqref{NLP} on the interval $[0,T]$. For such a control $u,$ the regularity of the solution $m$ shows that the cost functional $\mathcal J(m,u)$ defined in \eqref{CF} is finite, that is, $\mathcal J(m,u)< +\infty$. In this section, we prove that the functional \eqref{CF} achieves optimal value at some solution pair $(\widetilde{m},\widetilde{u})$ of the system \eqref{NLP}. 

\begin{proof}[Proof of Theorem \ref{EOC}]
The proof follows from convexity and lower semi-continuity of the  non-negative cost functional $\mathcal J(\cdot,\cdot)$. Since the functional $\mathcal J(\cdot, \cdot)$ is bounded below, there exists a constant $\alpha\geq 0$ and a minimizing sequence $\big\{(m_n,u_n)\big\} \subset \mathcal{A}$ such that
	\begin{equation*}
		\lim_{n \to \infty} \mathcal J(m_n,u_n)= \inf_{(m,u) \in \mathcal{A}} \mathcal J(m,u)=\alpha.	
	\end{equation*}
The pair $(m_n,u_n)$ is a regular solution of the system 
\begin{equation}\label{EOPME}
	\begin{cases}
		(m_n)_t= m_n \times (\Delta m_n +u_n)-  m_n \times (m_n \times (\Delta m_n +u_n))\ \ \ \text{in}\  \Omega_T, \\
		\frac{\partial m_n}{\partial \eta}=0 \ \ \  \ \ \ \ \ \ \ \text{in}\ \partial\Omega_T,\\
		m_n(\cdot,0)=m_0 \ \ \text{in} \ \Omega.
	\end{cases}	
\end{equation}
	As the set of admissible class of controls $\mathcal{U}_{ad}$ is a closed, convex and bounded subset of the reflexive Banach space $L^2(0,T;H^1(\Omega))$, it is weakly sequentially compact. There exists a sub-sequence again denoted as $\{u_n\}$ such that $u_n \rightharpoonup \widetilde{u}$ weakly in $\mathcal{U}_{ad}$ for some element $\widetilde{u} \in \mathcal{U}_{ad}$.  From   Theorem \ref{GLOBAL}, $\{m_n\}$ is uniformly bounded in $L^2(0,T;H^3(\Omega)) \cap C([0,T];H^2(\Omega))$ and $\left\{ (m_n)_t\right\}$ is uniformly bounded in $L^2(0,T;H^1(\Omega))$. Then, by Aubin–Lions–Simon compactness theorem, $\{m_n\}$ is relatively compact in $C([0,T];H^1(\Omega)) \cap L^2(0,T;H^2(\Omega))$. Therefore, there exists a subsequence (again represented as) $\left\{ (m_n,u_n)\right\}$ such that
	\begin{eqnarray} \left\{\begin{array}{ccccl}
		u_n &\overset{w}{\rightharpoonup} & \widetilde{u} \ &\mbox{weakly in}&  L^2(0,T;H^1(\Omega)),\\
		m_n &\overset{w}{\rightharpoonup} & \widetilde{m} \  &\mbox{weakly in}&  L^2(0,T;H^3(\Omega)),\\
		(m_n)_t &\overset{w}{\rightharpoonup} & \widetilde{m}_t \ &\mbox{weakly in}&  L^2(0,T;H^1(\Omega)),\\
		m_n &\overset{s}{\to} & \widetilde{m} \  &\mbox{strongly in}& C([0,T];H^1(\Omega))\cap L^2(0,T;H^2(\Omega)),	 \ \ \mbox{as} \  \ n\to \infty. \label{P2}
		\end{array}\right.	
	\end{eqnarray}
We need to prove the following lemma to validate that $(\widetilde{m},\widetilde{u})$ is an optimal pair for (OCP).
	\begin{Lem}\label{lem2}
		Suppose the convergences in \eqref{P2} hold true. Then for any $v \in L^2(0,T;L^2(\Omega))$, we have
		\begin{enumerate}[label=(\roman*)]
			\item $\displaystyle \int_0^T \big(|\nabla m_n|^2m_n,v\big)\ dt \to \int_0^T \big(|\nabla \widetilde{m}|^2\widetilde{m},v\big)\ dt$,\vspace{0.1in}
			\item $\displaystyle \int_0^T \big(m_n \times \Delta m_n,v\big)\ dt \to \int_0^T \big(\widetilde{m} \times \Delta \widetilde{m},v\big)\ dt$,\vspace{0.1in}
			\item $\displaystyle\int_0^T \big(m_n \times u_n,v\big)\ dt  \to \int_0^T \big(\widetilde{m} \times \widetilde{u},v\big)\ dt$,\vspace{0.1in}
			\item $\displaystyle \int_0^T \big(m_n \times (m_n \times u_n),v\big)\ dt \to \int_0^T \big(\widetilde{m} \times (\widetilde{m} \times \widetilde{u}),v\big)\ dt$.
		\end{enumerate}
	\end{Lem}	
	The proof of Lemma \ref{lem2} is given after the completion of this theorem.\\ \\
	By taking $n\to \infty$ in \eqref{EOPME}, invoking \eqref{P2} and Lemma \ref{lem2}, we obtain that $(\widetilde{m},\widetilde{u})$ satisfies the system \eqref{EP} in $L^2(0,T;L^2(\Omega))$.

	Also, since for any $\left\{ (m_n,u_n)\right\} \subset \mathcal{A}$ with $m_n\rightharpoonup \widetilde{m}$ weakly in $L^2(0,T;L^2(\Omega)),$ $m_n(\cdot,T)\rightharpoonup \widetilde{m}(\cdot,T)$ weakly in $L^2(\Omega)$ and $u_n \rightharpoonup \widetilde{u}$ weakly in $L^2(0,T;H^1(\Omega))$, the functional $\mathcal J(\cdot,\cdot)$ is lower semi-continuous, that is,
	$$\mathcal J(\widetilde{m},\widetilde{u}) \leq \liminf_{n \to \infty} \mathcal J(m_n,u_n) < +\infty,$$
	whence $(\widetilde{m},\widetilde{u})$ is an admissible pair. As  $\left\{(m_n,u_n)\right\}$ is a minimizing sequence, we have
	\begin{equation}\label{I2}
		\mathcal J(\widetilde{m},\widetilde{u}) \leq \lim_{n \to \infty} \inf \mathcal J(m_n,u_n)=\lim_{n \to \infty} \mathcal J(m_n,u_n)=\alpha.
	\end{equation}
	Since $\alpha$ is the infimum of the functional $\mathcal J$ over $\mathcal{A}$, $\alpha \leq \mathcal J(\widetilde{m},\widetilde{u})$, and hence combining with \eqref{I2}, we get 
	$\displaystyle{\mathcal J(\widetilde{m},\widetilde{u}) = \alpha = \inf_{(m,u) \in \mathcal{A}} \mathcal J(m,u).}$
	This completes the proof.
\end{proof}
\begin{Rem}
Note that Theorem \ref{EOC} only shows the existence of a globally optimal control, but uniqueness of optimal control for (OCP) may not be possible since it is a non-convex optimal control problem. Therefore, it is possible that the (OCP) has more than one local or global optimal controls.
\end{Rem}
\begin{proof}[Proof of Lemma \ref{lem2}]
	The proof of (i) and (ii) are same as that of (i) and (ii) in Lemma \ref{lem1}. We shall prove (iii) and (iv). For the convergence of (iii), using  H\"older's inequality, continuous embedding $H^1(\Omega) \hookrightarrow L^4(\Omega)$ and using vector product property $v\cdot(\widetilde{m}\times (u_n -\widetilde{u}))=( v \times \widetilde{m}) \cdot (u_n-\widetilde{u})$, we get
	\begin{eqnarray*}
		\lefteqn{\int_0^T \Big[\big(m_n\times  u_n,v\big)-\big(\widetilde{m}\times \widetilde{u},v\big)\Big]\ dt= \int_0^T \Big[\big((m_n-\widetilde{m}) \times u_n,v\big) + \big(\widetilde{m} \times (u_n - \widetilde{u}),v\big)\Big]\ dt}\\
		&\leq&  \int_0^T \|m_n(t)-\widetilde{m}(t)\|_{L^4(\Omega)} \|u_n(t)\|_{L^4(\Omega)} \|v(t)\|_{L^2(\Omega)}\ dt   + \int_0^T \big(u_n-\widetilde{u},v \times \widetilde{m}\big)\ dt\\
		&\leq& C\ \|m_n-\widetilde{m}\|_{L^\infty(0,T;H^1(\Omega))} \|u_n\|_{L^2(0,T;H^1(\Omega))} \|v\|_{L^2(0,T;L^2(\Omega))}  + \int_0^T \big(u_n-\widetilde{u},v \times \widetilde{m}\big)\ dt\ \to 0, \ \mbox{as} \ n \to \infty.
	\end{eqnarray*}
The first term on the right-hand side of the above inequality tends to $0$ as $m_n\to\widetilde{m}$ strongly in $L^\infty(0,T;H^1(\Omega))$. Furthermore, since	$$\int_0^T \|v(t) \times \widetilde{m}(t)\|^2_{L^2(\Omega)} dt \leq \int_0^T \|\widetilde{m}(t)\|^2_{L^\infty(\Omega)} \|v(t)\|^2_{L^2(\Omega)}dt \leq C\  \|\widetilde{m}\|^2_{L^\infty(0,T;H^2(\Omega))}\|v\|^2_{L^2(0,T;L^2(\Omega))} <+ \infty,$$ and $u_n \rightharpoonup \widetilde{u}$ weakly in $L^2(0,T;L^2(\Omega))$, we have $\int_0^T \big(u_n-\widetilde{u}, v \times \widetilde{m}\big)\ dt \to 0$ as $n \to \infty$.  
	
Now, for the convergence of (iv), we proceed as follows:
	\begin{eqnarray*}
		\lefteqn{\int_0^T \big[\big(m_n\times (m_n \times u_n),v\big)-\big(\widetilde{m} \times (\widetilde{m} \times \widetilde{u}),v\big)\big]\ dt}\\
		&=&	 \int_0^T \Big[\big((m_n - \widetilde{m})\times (m_n \times u_n),v\big) + \big(\widetilde{m} \times ((m_n - \widetilde{m})\times u_n),v\big)+ \big(\widetilde{m} \times (\widetilde{m} \times (u_n-\widetilde{u})),v\big)\Big]\ dt\\
		&\leq& \int_0^T \|m_n(t)-\widetilde{m}(t)\|_{L^4(\Omega)} \|m_n(t)\|_{L^\infty(\Omega)} \|u_n(t)\|_{L^4(\Omega)} \|v(t)\|_{L^2(\Omega)}dt\\
		&& + \int_0^T \|\widetilde{m}(t)\|_{L^\infty} \|m_n(t)-\widetilde{m}(t)\|_{L^4(\Omega)} \|u_n(t)\|_{L^4(\Omega)} \|v(t)\|_{L^2(\Omega)} dt + \int_0^T \big(\widetilde{m} \times (\widetilde{m} \times (u_n-\widetilde{u})),v\big) \ dt\\
		&\leq& C\ \|m_n-\widetilde{m}\|_{L^\infty(0,T;H^1(\Omega))} \|m_n\|_{L^\infty(0,T;H^2(\Omega))} \|u_n\|_{L^2(0,T;H^1(\Omega))} \|v\|_{L^2(0,T;L^2(\Omega))}\\
		&&+ \int_0^T \big(u_n-\widetilde{u},\widetilde{m} \times (\widetilde{m} \times v)\big) \ dt \ \ \to 0,  \ \mbox{as} \ n \to \infty.
	\end{eqnarray*}
Indeed, the last integral follows from  the vector identity 
$$ (\widetilde{m} \times (\widetilde{m} \times (u_n-\widetilde{u})))\cdot v=-(\widetilde{m}\times v)\cdot (\widetilde{m}\times (u_n-\widetilde{u}))= (\widetilde{m} \times (\widetilde{m} \times v)) \cdot (u_n -\widetilde{u})$$
 and  for the convergence of this integral, we used  again the fact that
 $$ \int_0^T \|\widetilde{m}(t) \times \big(\widetilde{m}(t) \times v(t)\big)\|^2_{L^2(\Omega)} dt  \leq C\   \|\widetilde{m}\|^4_{L^\infty(0,T;H^2(\Omega))} \|v\|^2_{L^2(0,T;L^2(\Omega))}<+\infty,$$
to conclude $\int_0^T \big(u_n-\widetilde{u},\widetilde{m} \times (\widetilde{m} \times v)\big)\ dt \to 0$ as $n\to \infty.$  Hence the proof.
\end{proof} 
\section{First-Order Optimality Conditions}\label{SOC}
It is evident from Theorem \ref{GLOBAL} that the existence of a unique regular solution of \eqref{NLP} is proved when the control $u \in \mathcal{U}_{ad}$, where $\mathcal{U}_{ad}$ is a closed and bounded set in $L^2(0,T;H^1(\Omega))$. In this section, we prove the Fr\'echet derivative of a control-to-state operator which is merely defined on an open subset of $L^2(0,T;H^1(\Omega))$.

Let us consider the set
$$\mathcal{U}_R:= \left\{u\in L^2(0,T;H^1(\Omega))\ \big|\  \|u\|_{L^2(0,T;L^2(\Omega))}<R\right\}.$$ 
In view of estimate \eqref{MCOC}, it is clear that if $\|\nabla m_0\|^2_{L^2(\Omega)}+\|u\|^2_{L^2(0,T;L^2(\Omega))}< \frac{1}{8C^*}$, then Theorem \ref{GLOBAL} still holds true. The constant $R$ in $\mathcal{U}_R$ can be chosen as $R=\Big(\frac{1}{8C^*}-\|\nabla m_0\|^2_{L^2(\Omega)}\Big)^{1/2}.$  Moreover, $\mathcal{U}_R$ is an open ball in $L^2(0,T;H^1(\Omega))$ containing $\mathcal{U}_{ad}$.

We study the control-to-state operator $G:\mathcal{U}_R \to W^{1,2}(0,T;H^3(\Omega),H^1(\Omega))$ defined by $G(u)=m$.	
 To derive the Fr\'echet differentiability of this operator with respect to the control, we need to study the linearized system associated with \eqref{EP}. For arbitrary, but fixed $u \in \mathcal{U}_{R}$, let $m$ be the unique regular solution of \eqref{EP}. Consider the linearized system  given by
\begin{equation}\label{LS1}
	\begin{cases}
		v_t -  \Delta v - 2 m(\nabla m \cdot \nabla v)-  |\nabla m|^2v -  v \times \Delta m-  m \times \Delta v \\
		\hspace{0.5in}  -  v \times u +  v \times (m \times u) +  m \times (v \times u)=g\ \  \ \ \text{in} \ \ \ \Omega_T,\\ 	
		\frac{\partial v}{\partial \eta}=0 \ \ \ \ \ \  \text{on}  \  \ \ \partial \Omega_T,\\
		v(0)=v_0 \ \ \ \text{in} \ \ \ \Omega,
	\end{cases}
\end{equation}
where $g$ is any function in $L^2(0,T;H^1(\Omega))$ and $v_0 \in H^2(\Omega)$ satisfying $\frac{\partial v_0}{\partial \eta}=0$.

\begin{Thm}\label{TLS}
	Let $(m,u) \in W^{1,2}\big(0,T;H^3(\Omega),H^1(\Omega)\big)\ \times \  \mathcal{U}_{R}$ be the regular solution of \eqref{NLP}.	Then for any given $(g,v_0) \in L^2(0,T;H^1(\Omega))\ \times \  H^2(\Omega)$, the linearized system \eqref{LS1} has a unique regular solution $v \in W^{1,2}(0,T;H^3(\Omega),H^1(\Omega))$, which satisfies the following estimation:
	\begin{eqnarray*}
		\|v\|{_\mathcal M} &\leq& C \left(\|v_0\|_{H^2(\Omega)} + \ \|g\|_{L^2(0,T;H^1(\Omega))}\right)\\
		&&\times \exp\Big\{C  \left(T+\|m\|^2_{L^\infty(0,T;H^2(\Omega))}\big(\|m\|^2_{L^2(0,T;H^3(\Omega))}+ \|u\|^2_{L^2(0,T;H^1(\Omega))}\big)\right) \Big\}.
	\end{eqnarray*}
\end{Thm}
\begin{proof}[Proof:]
	We employ the  Galerkin approximation construction used in Theorem \ref{LOCAL} to write that of the following for \eqref{LS1}. Let $\{\xi_i\}_{i=1}^{\infty}$ be an orthonormal basis of $L^2(\Omega)$ consisting of eigenvectors for  $-\Delta+I$  with vanishing Neumann boundary condition. Suppose $W_n=\text{span}\{w_1,w_2,...,w_n\}$ and $\mathbb{P}_n:L^2\to W_n$ be the orthogonal projection. Consider the Galerkin system 
	\begin{equation}\label{LSGA}
		\begin{cases}
			(v_n)_t -  \Delta v_n -\mathbb{P}_n\Big[ 2 m(\nabla m \cdot \nabla v_n)+  |\nabla m|^2v_n +  v_n \times \Delta m +  m \times \Delta v_n\\
			\hspace{1in}  +  v_n \times u -  v_n \times (m \times u) -  m \times (v_n \times u)\Big]=\mathbb{P}_n\big(g\big)\ \ \text{in} \ \Omega_T,\\ 	
			v_n(0)=\mathbb{P}_n(v_0) \ \ \ \text{in} \ \Omega,
		\end{cases}
	\end{equation}
where $v_n=\sum_{k=1}^{n}c_{kn}(t) \xi_k$ and $\mathbb{P}_n(v_0)=\sum_{k=1}^{n} d_{kn}\xi_k$. Repeating the similar argument of Theorem \ref{LOCAL}, we can show that \eqref{LSGA} is equivalent to a system of $n$ linear ordinary differential equations in $n$ unknowns 
$c_{n}(t)=(c_{1n}(t),c_{2n}(t),\cdots,c_{nn}(t))^T$. By Theorem \ref{GLOBAL} and the assumption on $u \in L^2(0,T;H^1(\Omega))$, theory of ODEs yield a unique solution $c_n(t)$ on $[0,T]$ for each $n\in \mathbb{N}$. Hence the approximated Galerkin system \eqref{LSGA} has a unique solution $v_n\in C^1([0,T];W_n)$ on $\Omega \times [0,T]$.

By taking the $L^2$ inner product of \eqref{LSGA} with $v_n$ and using Lemma \ref{CPP}, we have
	\begin{eqnarray}
	\lefteqn{\frac{1}{2} \frac{d}{dt} \|v_n(t)\|^2_{L^2(\Omega)}+ \int_\Omega |\nabla v_n(t)|^2 dx = 2  \int_\Omega m\ (\nabla m \cdot \nabla v_n)\cdot v_n \ dx}\nonumber\\
	&&+\int_\Omega  |\nabla m|^2 |v_n|^2 dx +\int_\Omega   (m \times \Delta v_n)\cdot v_n\ dx - \int_\Omega  (m \times (v_n \times u))\cdot v_n\ dx +\int_\Omega  g\cdot v_n\ dx.\nonumber 
	\end{eqnarray}
	Using Young's inequality and the embeddings $H^1(\Omega) \hookrightarrow L^4(\Omega)$, $H^2(\Omega) \hookrightarrow L^\infty(\Omega)$, we obtain
	\begin{equation}\label{LSL2}
		\frac{1}{2} \frac{d}{dt} \|v_n(t)\|^2_{L^2(\Omega)}+  \frac{1}{2} \int_\Omega |\nabla v_n(t)|^2 dx  \leq \|g(t)\|^2_{L^2(\Omega)} + C \left(1+\|m(t)\|^2_{H^3(\Omega)} + \|u(t)\|^2_{H^1(\Omega)}\right)\ \|v_n(t)\|^2_{L^2(\Omega)}. 
	\end{equation}
	Taking $L^2$ inner product of \eqref{LSGA} with $\Delta^2v_n$, we have
	\begin{eqnarray}\label{LSH2E}
		\lefteqn{\frac{1}{2} \frac{d}{dt} \|\Delta v_n(t)\|^2_{L^2(\Omega)} +  \int_\Omega |\nabla \Delta v_n(t)|^2 dx = 2 \int_\Omega m(\nabla m \cdot \nabla v_n)\ \Delta^2 v_n\ dx}\nonumber\\
		&&+\ \int_\Omega |\nabla m|^2v_n\ \Delta^2v_n\ dx+ \int_\Omega (v_n \times \Delta m)\ \Delta^2v_n\ dx + \int_\Omega (m \times \Delta v_n) \ \Delta^2v_n\ dx+  \int_\Omega (v_n \times u)\ \Delta^2v_n\ dx\nonumber\\
		&&-\  \int_\Omega (v_n \times (m \times u))\ \Delta^2v_n\ dx - \int_\Omega (m \times (v_n \times u))\ \Delta^2v_n\ dx + \int_\Omega g\ \Delta^2v_n\ dx:= \sum_{\Gamma=1}^{8}\Gamma_i. 
	\end{eqnarray}
	Let us estimate the terms on the right hand side. For the first term $\Gamma_1$, doing an integration by parts, applying H\"older's inequality and the embeddings $H^1(\Omega) \hookrightarrow L^p(\Omega)\ \text{for} \ p \in [1,\infty]$ and $H^2(\Omega) \hookrightarrow L^\infty(\Omega)$, we get 
	\begin{flalign*}
	\Gamma_1 &=-2 \int_\Omega \left[\nabla m(\nabla m \cdot \nabla v_n)+m(D^2m\cdot\nabla v_n)+m(\nabla m\cdot D^2v_n)\right]\nabla \Delta v_n\ dx&\\
	&\leq \ 2\ \|\nabla m(t)\|_{L^4(\Omega)} \|\nabla m(t)\|_{L^8(\Omega)}  \|\nabla v_n(t)\|_{L^8(\Omega)} \|\nabla \Delta v_n(t)\|_{L^2(\Omega)}\\
	&+\ 2\ \|m(t)\|_{L^\infty(\Omega)} \left(\|D^2m(t)\|_{L^4(\Omega)} \|\nabla v_n(t)\|_{L^4(\Omega)}+\|\nabla m(t)\|_{L^\infty(\Omega)} \|D^2v_n(t)\|_{L^2(\Omega)} \right) \|\nabla \Delta v_n(t)\|_{L^2(\Omega)} \\
	&\leq C\ \|\nabla m(t)\|_{H^1(\Omega)} \|\nabla m(t)\|_{H^1(\Omega)}  \|\nabla v_n(t)\|_{H^1(\Omega)} \|\nabla \Delta v_n(t)\|_{L^2(\Omega)}\\
	&+\ C\ \|m(t)\|_{H^2(\Omega)}\left( \|D^2m(t)\|_{H^1(\Omega)} \|\nabla v_n(t)\|_{H^1(\Omega)}+\|\nabla m(t)\|_{H^2(\Omega)} \|D^2v_n(t)\|_{L^2(\Omega)}\right) \|\nabla \Delta v_n(t)\|_{L^2(\Omega)} \\
	&\leq \epsilon   \int_\Omega |\nabla \Delta v_n(t)|^2 dx +\ C(\epsilon)   \ \left(\|m(t)\|^4_{H^2(\Omega)}+ \|m(t)\|^2_{H^2(\Omega)} \|m(t)\|^2_{H^3(\Omega)}\right) \|v_n(t)\|^2_{H^2(\Omega)}.	
	\end{flalign*}
By proceeding in a similar way for the terms  $\Gamma_2$ and $\Gamma_3$, we derive
\begin{flalign*}
	\Gamma_2 &\leq \epsilon   \int_\Omega |\nabla \Delta v_n(t)|^2 dx+\ C(\epsilon) \left(\|m(t)\|^4_{H^2(\Omega)}+ \|m(t)\|^2_{H^2(\Omega)} \|m(t)\|^2_{H^3(\Omega)}\right) \|v_n(t)\|^2_{H^2(\Omega)},&\\
	\Gamma_3 &\leq \epsilon   \int_\Omega |\nabla \Delta v_n(t)|^2 dx +\ C(\epsilon)\  \|m(t)\|^2_{H^3(\Omega)} \|v_n(t)\|^2_{H^2(\Omega)}.	
\end{flalign*}
For the term $\Gamma_4$ and $\Gamma_5$, doing an integration by parts, using the property $(a \times b)\cdot b=0$ and proceeding as above, one can get that 
\begin{flalign*}
	\Gamma_4 &=- \int_\Omega \big( \nabla m \times \Delta v_n \big) \ \nabla \Delta v_n\ dx
	\leq \ \|\nabla m(t)\|_{L^\infty(\Omega)} \|\Delta v_n(t)\|_{L^2(\Omega)} \|\nabla \Delta v_n(t)\|_{L^2(\Omega)} &\\
	& \leq  \epsilon   \int_\Omega |\nabla \Delta v_n(t)|^2 dx   +\ C(\epsilon)\  \|m(t)\|^2_{H^3(\Omega)} \|v_n(t)\|^2_{H^2(\Omega)}, \\
	\Gamma_5 &=- \int_\Omega \big[ \nabla v_n \times u+ v_n \times \nabla u\big] \ \nabla \Delta v_n\ dx &\\
	&\leq \ \|\nabla v_n(t)\|_{L^4(\Omega)} \|u(t)\|_{L^4(\Omega)} \|\nabla \Delta v_n(t)\|_{L^2(\Omega)}+ \|v_n(t)\|_{L^\infty(\Omega)} \|\nabla u(t)\|_{L^2(\Omega)} \|\nabla \Delta v_n(t)\|_{L^2(\Omega)}\\
	&\leq \epsilon \int_\Omega |\nabla \Delta v_n(t)|^2 dx +\ C(\epsilon)\  \|u(t)\|^2_{H^1(\Omega)} \|v_n(t)\|^2_{H^2(\Omega)}.
\end{flalign*}
For the term $\Gamma_6$, an integration by parts followed by H\"older's inequality and continuous embeddings $H^1(\Omega) \hookrightarrow L^4(\Omega)$, $H^1(\Omega) \hookrightarrow L^8(\Omega)$ and $H^2(\Omega) \hookrightarrow L^\infty(\Omega)$, we derive
\begin{flalign*}
	\Gamma_6 &= \int_\Omega \Big[ \nabla v_n \times (m \times u)+v_n \times (\nabla m \times u)+v_n \times (m \times \nabla u)\Big]\ \nabla \Delta v_n\ dx&\\
	&\leq  \|\nabla v_n(t)\|_{L^4(\Omega)} \|m(t)\|_{L^8(\Omega)}  \|u(t)\|_{L^8(\Omega)} \|\nabla \Delta v_n(t)\|_{L^2(\Omega)}\\
	&+\  \|v_n(t)\|_{L^\infty(\Omega)} \left(\|\nabla m(t)\|_{L^4(\Omega)} \|u(t)\|_{L^4(\Omega)}+\| m(t)\|_{L^\infty(\Omega)} \|\nabla u(t)\|_{L^2(\Omega)} \right) \|\nabla \Delta v_n(t)\|_{L^2(\Omega)} \\
	&\leq \epsilon \int_\Omega |\nabla \Delta v_n(t)|^2 \ dx  + C(\epsilon)   \ \|m(t)\|^2_{H^2(\Omega)} \|u(t)\|^2_{H^1(\Omega)} \|v_n(t)\|^2_{H^2(\Omega)}.
\end{flalign*}
We can obtain an estimate similar to $\Gamma_6$ for $\Gamma_7$ as well. By substituting all these estimates in equation \eqref{LSH2E} and choosing a suitable value for $\epsilon$ and adding with \eqref{LSL2}, we get
	\begin{eqnarray}\label{LSE1}
		\lefteqn{\frac{d}{dt} \left(\|v_n(t)\|^2_{L^2(\Omega)}+\|\Delta v_n(t)\|^2_{L^2(\Omega)}\right)+  \|\nabla v_n(t)\|^2_{L^2(\Omega)}+\|\nabla \Delta v_n(t)\|^2_{L^2(\Omega)} } \nonumber\\
		&&\leq C\   \|g(t)\|^2_{H^1(\Omega)}+ C\ \bigg[1+\|m(t)\|^2_{H^2(\Omega)} \|m(t)\|^2_{H^3(\Omega)}  + \|m(t)\|^2_{H^2(\Omega)} \|u(t)\|^2_{H^1(\Omega)}\bigg]\ \|v_n\|^2_{H^2(\Omega)}.
	\end{eqnarray}
By invoking Lemma \ref{EN}, $\|v_n\|_{H^2(\Omega)}\leq C\ \left(\|v_n\|_{L^2(\Omega)}+\|\Delta v_n\|_{L^2(\Omega)}\right)$ and applying Gronwall's inequality, we obtain the uniform bounds for $\{v_n\}$ in $L^\infty(0,T;H^2(\Omega)):$
	\begin{equation}\label{LSE2}
		\sup_{t \in [0,T]} \left(\|v_n(t)\|^2_{L^2(\Omega)}+\|\Delta v_n(t)\|^2_{L^2(\Omega)}\right)
		\leq C_{\#}(\Omega,T,v_0,g,m,u),	
	\end{equation}
	where
	\begin{eqnarray*}
		C_\#(\Omega,T,v_0,g,m,u) \!\!&=&\!\! C\left(\|v_0\|^2_{L^2(\Omega)}+\|\Delta v_0\|^2_{L^2(\Omega)} +  \|g\|^2_{L^2(0,T;H^1(\Omega))}\right)\\
		&&\times exp\bigg\{C  \left(T+\|m\|^2_{L^\infty(0,T;H^2(\Omega))}\|m\|^2_{L^2(0,T;H^3(\Omega))}+\|m\|^2_{L^\infty(0,T;H^2(\Omega))} \|u\|^2_{L^2(0,T;H^1(\Omega))}\right) \bigg\}.
	\end{eqnarray*}
By integrating \eqref{LSE1} over $(0,t)$ 
and employing the uniform bounds for $v_n$ in $L^\infty(0,T;H^2(\Omega))$ from inequality \eqref{LSE2}, we have
\begin{equation}\label{LSE3}	
	\int_0^T \left(\|\nabla v_n(\tau)\|^2_{L^2(\Omega)}+\|\nabla \Delta v_n(\tau)\|^2_{L^2(\Omega)}\right)\ d\tau \leq C_\#(\Omega,T,v_0,g,m,u). 
\end{equation}
Hence, from \eqref{LSE2} and \eqref{LSE3}, $\{v_n\}$ is uniformly bounded in $L^\infty(0,T;H^2(\Omega)) \cap L^2(0,T;H^3(\Omega))$. Taking $H^1$ norm of $\{(v_n)_t\}$ in equation \eqref{LSGA}, using the estimates from Section \ref{FS} and substituting the uniform bounds for $\{v_n\}$ from estimates \eqref{LSE2} and \eqref{LSE3}, we derive
	$$\int_0^T \left( \|(v_n)_t(\tau)\|^2_{L^2(\Omega)} + \| \nabla (v_n)_t(\tau)\|^2_{L^2(\Omega)}\right)\ d\tau \leq C_{\#}(\Omega,T,v_0,g,m,u).$$
	Therefore, we also get that $\{(v_n)_t\}$ is uniformly bounded in $L^2(0,T;H^1(\Omega))$. By Aloglu weak$^*$ compactness and reflexive weak compactness theorems (Theorem 4.18, \cite{JCR}), we have	
	\begin{eqnarray}\nonumber
		\left\{\begin{array}{cccll}
			v_n &\overset{w}{\rightharpoonup} & v \  &\mbox{weakly in}& \ L^2(0,T;H^3(\Omega)),\\
			v_n &\overset{w^\ast}{\rightharpoonup} & v \ &\mbox{weak$^*$ in}& \ L^{\infty}(0,T;H^2(\Omega)),\\
			(v_n)_t &\overset{w}{\rightharpoonup} & v_t \ &\mbox{weakly in}& \ L^2(0,T;H^1(\Omega)), \ \ \mbox{as} \  \ n\to \infty. 
		\end{array}\right.	
	\end{eqnarray}
	
	Again as a result of Aubin-Lions-Simon lemma (see, Corollary 4, \cite{JS}), we can get a sub-sequence of $\{v_n\}$ such that $v_n \overset{s}{\to} v$ strongly in $L^2(0,T;H^2(\Omega))$ and $C([0,T];H^1(\Omega))$. Using these strong convergence and results similar to Lemma \ref{lem1}, we can show $v$ is indeed a regular solution of \eqref{LS1}. Also, since the problem is linear, uniqueness can be directly shown by setting $v_1-v_2=w$, where $v_1$ and $v_2$ are regular solutions of \eqref{LS1} and deriving an $L^2$ estimate for $w$ as in \eqref{LSE1} followed by the application of Gronwall's inequality. Hence the proof.
\end{proof}

Before stating the optimality conditions satisfied by $(\widetilde{m},\widetilde{u})$, we analyze the differentiability of the control-to-state operator.
\begin{Pro}{(Fr\'echet differentiability of control-to-state map)}\label{CTSM} If system~\eqref{EP} has a regular solution for some $w \in \mathcal{U}_{R}$, then there exists an open neighbourhood $\mathcal{U}$ of $w$ in $\mathcal{U}_{R}$ such that for any $u \in \mathcal{U}$ and initial data $m_0$, we have a regular solution $m_u$ in $W^{1,2}(0,T;H^3(\Omega),H^1(\Omega))$. Also, the control-to-state map $G:\mathcal{U} \to W^{1,2}(0,T;H^3(\Omega),H^1(\Omega))$ defined by $G(u)=m_u$ is of class $C^\infty$.	Moreover, if $z=D G(u)\cdot h$, for some $u \in \mathcal{U}_R$ and some $h \in L^2(0,T;H^1(\Omega))$, then $z\in W^{1,2}(0,T;H^3(\Omega),H^1(\Omega))$ is a unique regular solution of the following linearized system: 
	\begin{equation}\label{LS2}
		\begin{cases}
			z_t -  \Delta z - 2 m(\nabla m \cdot \nabla z)-  |\nabla m|^2z -  z \times \Delta m-  m \times \Delta z -  z \times u\\
			\hspace{0.3in}  +\  z \times (m \times u) +  m \times (z \times u)- m \times h +  m \times (m \times h)=0 \ \  \ \text{in} \ \ \Omega_T,\\
			\frac{\partial z}{\partial \eta}=0 \ \ \ \ \text{on} \ \partial \Omega_T,\\
			z(0)=0 \ \ \text{in} \ \Omega.	
		\end{cases}
	\end{equation}
\end{Pro}
\begin{proof}[Proof:]
 Consider a map $F:\mathcal M \times \mathcal U_R \to L^2(0,T;H^1(\Omega)) \times H^2(\Omega)$ 
defined by $$ F(m,u)= \Big(	m_t -  \Delta m -  |\nabla m|^2 m -  m \times \Delta m -  m \times u +  m \times (m \times u), \ m(0)-m_0\Big),$$ where we recall that $\mathcal M=W^{1,2}(0,T;H^3(\Omega),H^1(\Omega)).$
Before going  to prove the Fr\'echet differentiability of the control-to-state operator $G,$  we need to prove that of the map $F.$   The mapping 
$(m,u) \longmapsto \left( m_t-\Delta m, m(0) \right)$
is linear and bounded from $ \mathcal M \times \mathcal U_R  \to L^2(0,T;H^1(\Omega)) \times H^2(\Omega)$.

In order to obtain the Fr\'echet differentiability of the other nonlinear terms $|\nabla m|^2 m, \  m \times \Delta m, \   m \times u$ and $ m \times (m \times u)$, we estimate each terms as follows. By applying H\"older's inequality and the embeddings $ H^1(\Omega) \hookrightarrow L^4(\Omega),H^2(\Omega) \hookrightarrow L^\infty(\Omega),$  we get
\begin{flalign*}
	F_1(m,\varphi,u,\theta) &= \left\| |\nabla (m+\varphi)|^2 (m+\varphi) - |\nabla m|^2\ m - |\nabla m|^2\ \varphi -2m\ (\nabla m\cdot  \nabla \varphi) \right\|_{L^2(0,T;L^2(\Omega))}&\\
	&=	\left\| m\ |\nabla \varphi|^2 + \varphi\ |\nabla \varphi|^2 +2\varphi\  (\nabla m\cdot\nabla \varphi)\right\|_{L^2(0,T;L^2(\Omega))}\\
	&\leq C\  \|m\|_{L^\infty(0,T;H^2(\Omega))} \|\varphi\|_{L^\infty(0,T;H^2(\Omega))} \|\varphi\|_{L^2(0,T;H^2(\Omega))}+C \ \|\varphi\|^2_{L^\infty(0,T;H^2(\Omega))} \|\varphi\|_{L^2(0,T;H^2(\Omega))}\\
	&\leq C\  \|m\|_{L^\infty(0,T;H^2(\Omega))}  \|(\varphi,\theta)\|^2_{\mathcal{M}\times L^2(0,T;H^1(\Omega))}+C\ \|(\varphi,\theta)\|^3_{\mathcal{M}\times L^2(0,T;H^1(\Omega))},
\end{flalign*}
\begin{flalign*}
	F_2(m,\varphi,u,\theta) &= \left\| \ \big( (m +\varphi)\times \Delta (m+\varphi)\big) -m \times \Delta m - m \times \Delta \varphi- \varphi \times \Delta m\right\|_{L^2(0,T;L^2(\Omega))}&\\
	 &= \left\| \varphi \times \Delta \varphi \right\|_{L^2(0,T;L^2(\Omega))} \leq C\  \|\varphi\|_{L^\infty(0,T;H^2(\Omega))} \|\varphi\|_{L^2(0,T;H^2(\Omega))}\leq C\ \|(\varphi,\theta)\|^2_{\mathcal{M}\times L^2(0,T;H^1(\Omega))}.
\end{flalign*}
For the control terms, we  obtain that
\begin{flalign*}
	F_3(m,\varphi,u,\theta) &=\left\| \ \big( (m + \varphi)\times (u+\theta)\big)- m\times u -\varphi \times u -m \times \theta \right\|_{L^2(0,T;L^2(\Omega))}&\\
	&= \left\| \varphi \times \theta \right\|_{L^2(0,T;L^2(\Omega))} \leq C\ \|\varphi\|_{L^\infty(0,T;H^1(\Omega))} \|\theta\|_{L^2(0,T;H^1(\Omega))}\leq C\ \|(\varphi,\theta)\|^2_{\mathcal{M}\times L^2(0,T;H^1(\Omega))}
\end{flalign*}
and
\begin{flalign*}
	F_4(m,\varphi,u,\theta) &= \left\| (m+\varphi) \times \big( (m+\varphi) \times (u+\theta)\big) - m \times (m \times u) -\varphi \times (m \times u)\right.&\\
	& \hspace{.5in}\left.- m \times (\varphi \times u) -m \times (m \times \theta)\right\|_{L^2(0,T;L^2(\Omega))}\\
	 &= \left\| m \times (\varphi \times \theta) + \varphi \times (m \times \theta) + \varphi \times (\varphi \times u) + \varphi \times (\varphi \times \theta) \right\|_{L^2(0,T;L^2(\Omega))}\\
	& \leq C\ \|\varphi\|_{L^\infty(0,T;H^2(\Omega))} \|\varphi\|_{L^\infty(0,T;H^1(\Omega))}\left(\|u\|_{L^2(0,T;H^1(\Omega))} +  \|\theta\|_{L^2(0,T;H^1(\Omega))}\right)\\
	& \ \ \ \ \ \ +C\ \|m\|_{L^\infty(0,T;H^2(\Omega))} \|\varphi\|_{L^\infty(0,T;H^1(\Omega))} \|\theta\|_{L^2(0,T;H^1(\Omega))}\\
	&\leq C\ \|m\|_{L^\infty(0,T;H^2(\Omega))}  \|(\varphi,\theta)\|^2_{\mathcal{M}\times L^2(0,T;H^1(\Omega))}+C\ \|(\varphi,\theta)\|^3_{\mathcal{M}\times L^2(0,T;H^1(\Omega))}\\
	&\ \ \ \  \ \ +C\ \|u\|_{L^2(0,T;H^1(\Omega))}  \|(\varphi,\theta)\|^2_{\mathcal{M}\times L^2(0,T;H^1(\Omega))}.
\end{flalign*}
Now, dividing each $F_i(\cdot)$ by $\|(\varphi,\theta)\|_{\mathcal{M}\times L^2(0,T;H^1(\Omega))},$ we can directly see that $$\frac{F_i(m,\varphi,u,\theta)}{\|(\varphi,\theta)\|_{\mathcal{M}\times L^2(0,T;H^1(\Omega))}} \to 0 \ \  \mbox{as} \ \  \|(\varphi,\theta)\|_{\mathcal{M}\times L^2(0,T;H^1(\Omega))} \to 0, \ i=1,\cdots, 4.$$
Therefore, $F(\cdot,\cdot)$ is Fr\'echet differentiable on $\mathcal M \times \mathcal U_R.$ In fact, we can show that $F$ is of class $C^\infty$. Moreover, 
\begin{align*}
	&\frac{\partial F}{\partial m} (m,u)\cdot v= \Big(v_t -  \Delta v - 2 m(\nabla m \cdot \nabla v)-  |\nabla m|^2v -  v \times \Delta m\nonumber\\ 
	&\hspace{1in}-  m \times \Delta v -  v \times u +  v \times (m \times u)	+  m \times (v \times u), v(0) \Big).
\end{align*}
It is evident that $\frac{\partial F}{\partial m} (m,u)\cdot v =(g,v_0)$ with $(g,v_0)\in L^2(0,T;H^1(\Omega))\times H^2(\Omega), \frac{\partial v_0}{\partial \eta}=0$ if and only if $v$ solves the system \eqref{LS1}.

From the existence and uniqueness result (Theorem \ref{TLS}) of the linearized system \eqref{LS1}, it is clear that $\frac{\partial F}{\partial m}(m_u,u)$ is an isomorphism from $\mathcal M$ onto $L^2(0,T;H^1(\Omega)) \times H^2(\Omega)$ for every $(m,u)\in \mathcal M\times \mathcal U_{R}.$  If $m_w$ is a regular solution of the system \eqref{NLP} corresponding to the control $w\in\mathcal U_R,$ we have $F(m_w,w)=(0,0).$ Applying the implicit function theorem, we deduce that there exists an open neighborhood $\mathcal{U} \subset \mathcal U_R$ of $w$ and a mapping $G:\mathcal{U} \to \mathcal M$ defined by $G(u)=m_u$ such that 
$$F(G(u),u)=(0,0) \ \ \ \text{for every}\  u \in \mathcal{U},$$
and $G$ is of class $C^\infty$. Hence, taking Gateaux derivative of $F$ with respect to $u$ using chain rule, we obtain
\begin{equation*}
D_mF(G(u),u) \circ [D_uG(u)\cdot h] +D_uF(G(u),u)\cdot h =(0,0) \ \ \ \ \forall \ h \in L^2(0,T;H^1(\Omega)).	
\end{equation*}
Further, note that 
$ D_uF(m_u,u)\cdot h = \big(-m \times h+m \times (m \times h),0\big).$
By setting $D_uG(u)\cdot h=z$ and taking note of  $g=[m \times h-m \times (m \times h)]\in L^2(0,T;H^1(\Omega)),$ we conclude that  $D_mF(G(u),u)\cdot z =(m \times h-m \times (m \times h),0)$  if and only if $z$ is a regular solution of the linearized system \eqref{LS2} by Theorem \ref{TLS}.  Hence the proof.
\end{proof}

\begin{Rem}
	By extending the value of $R$ in the control set $\mathcal{U}_R$ and applying Proposition \ref{CTSM} on this modified set, we can conclude that the set of controls in $L^2(0,T;H^1(\Omega))$ for which there exists a regular solution in $\mathcal{M}$ forms an open set.
\end{Rem}

Next, we  study the solvability of the adjoint problem. While obtaining the first-order necessary conditions, instead of working with the strong solution of the adjoint equation, we will work with the weak one. So, a weak formulation of the adjoint problem \eqref{AP} is given below. Suppose $\langle\cdot,\cdot \rangle$ denotes the inner product between $H^1(\Omega)$ and $H^1(\Omega)^*$.
\begin{Def}\label{WFAP}
A function $\phi\in L^2(0,T;H^1(\Omega))$ with $\phi_t \in  L^2(0,T;H^1(\Omega)^*)$ is a weak solution of the adjoint system \eqref{AP} if for each $\psi\in H^1(\Omega),$ the following hold:
\begin{align*} 
& (i) \ \ \	-\langle \phi_t , \psi\rangle +(\nabla \phi, \nabla \psi) = -(\nabla(\phi \times \widetilde{m}),\nabla \psi)+(\Delta \widetilde{m}\times \phi,\psi)+(\widetilde{u}\times \phi,\psi) +2((\widetilde{m}\cdot \phi)\nabla \widetilde{m},\nabla \psi)\\\vspace{.1in} 
	&\hspace{.75in}+ (|\nabla \widetilde{m}|^2\phi,\psi)+((\phi \times \widetilde{m})\times \widetilde{u}, \psi)+(\phi \times (\widetilde{m}\times \widetilde{u}),\psi)+(\widetilde{m}-m_d,\psi)\ \ \text{for a.e.} \ t\in [0,T], \\\vspace{.1in} 
& (ii)	\ \ \	\phi(T)=\widetilde{m}(T)-m_\Omega \ \ \ \text{in} \ \ \Omega.
\end{align*}
\end{Def}

\begin{Thm}\label{TAS}
	Let $(\widetilde{m},\widetilde{u}) \in W^{1,2}(0,T;H^3(\Omega),H^1(\Omega)) \times \mathcal{U}_R$ be the regular solution pair of \eqref{NLP}. Then the adjoint system \eqref{AP} has a unique weak solution $\phi \in W^{1,2}(0,T;H^1(\Omega),H^1(\Omega)^*)$.  Moreover, the adjoint state $\phi$ satisfies the following estimation:
	\begin{eqnarray}
	\lefteqn{\|\phi\|^2_{L^2(0,T;H^1(\Omega))} +\|\phi_t\|^2_{L^2(0,T;H^1(\Omega)^*)} } \nonumber\\
	&\leq& C\Big(\|\widetilde{m}(T)-m_\Omega\|^2_{L^2(\Omega)} +\|\widetilde{m}-m_d\|^2_{L^2(0,T;L^2(\Omega))}  \Big)\label{AEEE}	\\
		&&\times\ \|\widetilde{m}\|^4_{L^\infty(0,T;H^2(\Omega))}\exp{\left\{C \int_0^T \left(\|\widetilde{m}(t)\|^2_{H^3(\Omega)}+ \|\widetilde{u}(t)\|^2_{H^1(\Omega)}\right)dt\right\}}.\nonumber
	\end{eqnarray}
\end{Thm}
\begin{proof}[Proof:]
The proof follows the Galerkin method used in  Theorem \ref{TLS}.  
Recall that the eigenfunctions $\{w_j\}_{j=1}^{\infty}$ of the operator $-\Delta +I$  forms an orthonormal basis in $L^2(\Omega)$ and orthogonal basis in $H^1(\Omega)$ and $H^2(\Omega)$.
For each $n$, we want to find a solution $\phi_n=\sum_{j=1}^{n}g_{jn}(t)w_j$ of the following approximated system for $j=1,\cdots,n:$
\begin{equation}\label{APGS}
	\begin{cases}
		-\big(\phi_n'(t),w_j\big)+\big(\nabla\phi_n(t),\nabla w_j\big)=-\big(\nabla  (\phi_n(t) \times \widetilde{m}(t)),\nabla w_j\big)+\big((\Delta \widetilde{m}(t)\times \phi_n(t)),w_j\big)\\
		\hspace{0.5in}+\ \big((\widetilde{u}(t)\times  \phi_n(t)),w_j\big)+2 \big((\widetilde{m}(t)\cdot \phi_n(t))\nabla \widetilde{m}(t),\nabla w_j\big)+\big(|\nabla \widetilde{m}(t)|^2\phi_n(t),w_j\big)\\
		\hspace{0.5in}+\  \big((\phi_n(t)\times \widetilde{m}(t))\times \widetilde{u}(t),w_j\big)+\big(\phi_n(t) \times (\widetilde{m}(t)\times \widetilde{u}(t)),w_j\big)+\big(\widetilde{m}(t)-m_d(t),w_j\big), \\
		\phi_n(T)=\mathbb{P}_n\big(\widetilde{m}(T)-m_\Omega\big).
	\end{cases}
\end{equation}

The system \eqref{APGS} is equivalent to a system of linear ODEs for the functions $g_{1n},\cdots,g_{nn}$. The solvability of the ODEs and \eqref{APGS} follow from a similar argument to Theorem \ref{LOCAL}. 

Multiplying \eqref{APGS} by $g_{jn}(t)$ and summing over $j=1,...,n,$ we get
\begin{eqnarray*}
	\lefteqn{-\frac{1}{2}\frac{d}{dt}\|\phi_n(t)\|^2_{L^2(\Omega)}+ \int_\Omega |\nabla \phi_n(t)|^2 dx = -\int_\Omega \nabla(\phi_n \times \widetilde{m})\cdot \nabla\phi_n \ dx+\int_\Omega |\nabla \widetilde{m}|^2 |\phi_n|^2\ dx}\\
	&&+2 \int_\Omega \big( (\widetilde{m}\cdot \phi_n)\nabla \widetilde{m}\big)\cdot \nabla \phi_n \ dx + \int_\Omega \big((\phi_n \times \widetilde{m}) \times \widetilde{u}\big)\cdot \phi_n \ dx + \int_\Omega (\widetilde{m}-m_d)\cdot \phi_n \ dx,
\end{eqnarray*}	 
where we also employed the property $a\cdot (a \times b)=0$. Now, applying H\"older's inequality and embedding $H^1(\Omega)\hookrightarrow L^4(\Omega)$, $H^2(\Omega) \hookrightarrow L^\infty(\Omega)$ and \eqref{ES4}, we have
\begin{eqnarray*}
	\lefteqn{-\frac{1}{2}\frac{d}{dt}\|\phi_n(t)\|^2_{L^2(\Omega)}+ \frac{1}{2}\int_\Omega |\nabla \phi_n(t)|^2 dx} \\
	&&\leq   \|\widetilde{m}(t)-m_d(t)\|^2_{L^2(\Omega)}+\ C\ \left(\|\widetilde{m}(t)\|^2_{H^3(\Omega)} +\|\widetilde{u}(t)\|^2_{H^1(\Omega)}\right)  \|\phi_n(t)\|^2_{L^2(\Omega)}. \nonumber
\end{eqnarray*}
By taking integration from t to T and then applying Gronwall's inequality, followed by the inequality $\|\phi_n(T)\|_{L^2(\Omega)}\leq \|\phi(T)\|_{L^2(\Omega)}$, we derive
\begin{align}
	\|\phi_n(t)\|^2_{L^2(\Omega)}+\int_t^T\int_\Omega |\nabla \phi_n(s)|^2 dxds &\leq 2\Big(\|\widetilde{m}(T)-m_\Omega\|^2_{L^2(\Omega)} + \|\widetilde{m}-m_d\|^2_{L^2(0,T;L^2(\Omega))} \Big) \label{AEE1}\\
	&\hspace{0.1in}\times \exp{\left\{C \int_0^T \left(\|\widetilde{m}(s)\|^2_{H^3(\Omega)}+ \|\widetilde{u}(s)\|^2_{H^1(\Omega)}\right)ds\right\}}, \ \ \forall t\in [0,T).\nonumber
\end{align}
From estimate \eqref{AEE1}, it is clear that $\{\phi_n\}$ is uniformly bounded in $L^\infty(0,T;L^2(\Omega)) \cap L^2(0,T;H^1(\Omega))$.

 Next, we obtain the bound for $\{\phi_n^\prime(t)\}.$ Fix any $v\in H^1(\Omega)$ with $\|v\|_{H^1(\Omega)}\leq 1$. We can split $v$ as $v=v_1+v_2$, where $v_1\in \text{span}\{w_k\}_{k=1}^n$ and $(v_2,w_k)=0, k=1,...,n$. Since the functions $\{w_k\}_{k=1}^{\infty}$ are orthogonal in $H^1(\Omega)$, $\|v_1\|_{H^1(\Omega)}\leq \|v\|_{H^1(\Omega)} \leq 1$. Taking $w_j=v_1$ in equation \eqref{APGS} and applying  H\"older's inequality, the embeddings $H^1(\Omega)\hookrightarrow L^p(\Omega)$ for $p\in [1,\infty)$ and the fact that $|\widetilde{m}|=1$, we estimate the right-hand side terms of \eqref{APGS}:
\begin{flalign*}
	\big(\nabla (\phi_n(t)\times \widetilde{m}(t)), \nabla v_1\big)&\leq \|\widetilde{m}(t)\|_{L^\infty(\Omega)}\|\nabla \phi_n(t)\|_{L^2(\Omega)} \|\nabla v_1\|_{L^2(\Omega)}+ \|\phi_n(t)\|_{L^4(\Omega)}\|\nabla \widetilde{m}(t)\|_{L^4(\Omega)}\|\nabla v_1\|_{L^2(\Omega)}\\
	&\leq C\|\widetilde{m}(t)\|_{H^2(\Omega)}\|\phi_n(t)\|_{H^1(\Omega)}\|v_1\|_{H^1(\Omega)}, \\
		\big(|\nabla \widetilde{m}(t)|^2\phi_n(t),v_1\big)&\leq \|\nabla \widetilde{m}(t)\|^2_{L^4(\Omega)} \|\phi_n(t)\|_{L^4(\Omega)}\|v_1\|_{L^4(\Omega)}
	 \leq C\|\widetilde{m}(t)\|^2_{H^2(\Omega)}\|\phi_n(t)\|_{H^1(\Omega)}\|v_1\|_{H^1(\Omega)}.
\end{flalign*}
The same bound holds for the terms $(\Delta \widetilde{m}\times \phi_n,v_1)$ and $((\widetilde{m}\cdot \phi_n)\nabla\widetilde{m}(t),\nabla v_1).$ Further, we have the following estimates for the control terms:
\begin{flalign*}
	\big((\widetilde{u}(t)\times  \phi_n(t)),v_1\big)&\leq \|\widetilde{u}(t)\|_{L^4(\Omega)}\|\phi_n(t)\|_{L^2(\Omega)}\|v_1\|_{L^4(\Omega)}\\
	&\leq C\|\widetilde{u}(t)\|_{H^1(\Omega)}\|\phi_n(t)\|_{L^2(\Omega)}\|v_1\|_{H^1(\Omega)},\\
\big((\phi_n(t)\times \widetilde{m}(t))\times \widetilde{u}(t),v_1\big)&\leq \|\phi_n(t)\|_{L^2(\Omega)}\|\widetilde{m}(t)\|_{L^\infty(\Omega)}\|\widetilde{u}(t)\|_{L^4(\Omega)} \|v_1\|_{L^4(\Omega)}\\
	& \leq C\|\widetilde{u}(t)\|_{H^1(\Omega)}\|\phi_n(t)\|_{L^2(\Omega)}\|v_1\|_{H^1(\Omega)}.
\end{flalign*}
By combining all the above estimates, using $\langle \phi_n',v\rangle =\langle \phi_n',v_1\rangle $ and $\|v_1\|_{H^1(\Omega)}\leq 1$, we obtain from \eqref{APGS} that
$$|\langle \phi_n',v\rangle| \leq C\left(\|\widetilde{m}(t)\|^2_{H^2(\Omega)}\|\phi_n(t)\|_{H^1(\Omega)}+\|\widetilde{u}(t)\|_{H^1(\Omega)}\|\phi_n(t)\|_{L^2(\Omega)}+\|\widetilde{m}(t)-m_d(t)\|_{L^2(\Omega)}\right).$$
As the above estimate holds for every $v \in H^1(\Omega)$, we obtain upon integration on $[0,T]$ that
\begin{align}
	&\int_0^T \|\phi_n'(t)\|^2_{H^1(\Omega)^*} dt \leq C\left(\|\widetilde{m}\|^4_{L^\infty(0,T;H^2(\Omega))} \|\phi_n\|^2_{L^2(0,T;H^1(\Omega))}+ \|\widetilde{u}\|^2_{L^2(0,T;H^1(\Omega))} \|\phi_n\|^2_{L^\infty(0,T;L^2(\Omega))}\right.\nonumber\\
	&\hspace{2in}\left.+\|\widetilde{m}-m_d\|^2_{L^2(0,T;L^2(\Omega))}\right).&\label{AEE3}
\end{align}
Therefore, $\{\phi_n^\prime\}$ is uniformly bounded in $L^2(0, T;H^1(\Omega)^*)$. In view of \eqref{AEE1} and \eqref{AEE3}, appealing to Aloglu weak* compactness and reflexive weak compactness theorems, we have
 \begin{eqnarray} \left\{\begin{array}{cccll}
 		\phi_n &\overset{w}{\rightharpoonup} & \phi \  &\mbox{weakly in}& \ L^2(0,T;H^1(\Omega)),\\
 		\phi_n &\overset{w^\ast}{\rightharpoonup} & \phi \ &\mbox{weak$^*$ in}& \ L^{\infty}(0,T;L^2(\Omega)),\\
 		\phi^\prime_n &\overset{w}{\rightharpoonup} & \phi^\prime \ &\mbox{weakly in}& \ L^2(0,T;H^1(\Omega)^*), \ \ \mbox{as} \  \ n\to \infty. \label{AEWC}
 	\end{array}\right.	
 \end{eqnarray}

The Aubin-Lions-Simon lemma (see, Corollary 4, \cite{JS}) establishes the existence of a sub-sequence of $\{\phi_n\}$ (again denoted as $\{\phi_n\}$) such that $\phi_n \overset{s}{\to} \phi$ strongly in $L^2(0,T;L^2(\Omega))$. Using this strong convergence along with weak-weak* convergences from \eqref{AEWC}, we can verify  that Definition \ref{WFAP}-(i)  holds true for every $\psi \in \text{span}(w_1,w_2,...)$. Further, as such functions are dense in $H^1(\Omega)$, it  holds true for every $\psi \in H^1(\Omega)$, almost every $t\in [0,T]$.  
 
Since $\phi\in  L^2(0,T;H^1(\Omega)),\phi_t\in  L^2(0,T;H^1(\Omega)^*),$ we infer that $\phi\in  C([0,T];L^2(\Omega)).$ It is sufficient to verify the terminal condition $\phi(T)=\widetilde{m}(T)-m_\Omega$ in a standard way. Finally, combining the estimates \eqref{AEE1} and \eqref{AEE3}, and using the sequential lower semi-continuity of $\phi_n$, we obtain the estimate \eqref{AEEE}. The uniqueness of weak solutions of the linear adjoint system \eqref{AP}  follows from an estimate similar to \eqref{AEE1}.   The proof is thus completed.
\end{proof}

\subsection{Optimality Conditions}\label{SFOOC}
Let $G:\mathcal U\to \mathcal M$ be the control-to-state operator defined by $G(u)=m_u$ as in Proposition \ref{CTSM}, where $m_u\in\mathcal M$ is a regular solution of the system \eqref{NLP} associated with the control $u.$  
Define a \emph{ reduced cost functional }  $\mathcal{I}  : \mathcal{U} \to \mathbb{R}$ by $\mathcal{I}(u)=\mathcal J(G(u),u)$. The optimal control problem (OCP) can be redefined in terms of the reduced functional as follows:
\begin{equation*}\text{(MOCP)}
\left\{	\begin{array}{lclclc}
		\text{minimize} \ \ \mathcal{I}(u)	\\
		u\in \mathcal U_{ad}.
	\end{array}
\right.
\end{equation*}
Next, we prove the first-order necessary optimality condition for the modified problem (MOCP) given by the variational inequality. To characterize the optimality condition in a concise structure,  we employ the classical adjoint problem approach. 
\begin{proof}[Proof of Theorem \ref{FOOC}:]
Let $\widetilde{u} \in \mathcal{U}_{ad}$ be an optimal control of (MOCP) with associated state $\widetilde{m}.$ For any optimal solution $\widetilde{u},$ the functional $\mathcal{I}(\cdot)$ must satisfy  the following inequality:
\begin{equation}\label{LI}
	D_u\mathcal{I}(\widetilde{u})(u-\widetilde{u})= \lim_{\epsilon \to 0} \frac{\mathcal{I}(\widetilde{u} + \epsilon (u - \widetilde{u}))-\mathcal{I}(\widetilde{u})}{\epsilon} \geq 0, \ \ \ \  \ \forall \ u \in \mathcal{U}_{ad},	
\end{equation}
since the control-to-state operator $G(\cdot)$ is Fr\'echet differentiable by Proposition \ref{CTSM}, Fr\'echet differentiability of the functional $\mathcal{I(\cdot)}$ follows by the chain rule.  Setting $h=u-\widetilde{u}$, it is easy to see from the definition of $\mathcal J(\cdot,\cdot)$ that
\begin{flalign}\label{FD1}
D_u \mathcal{I} (\widetilde{u})\cdot h &= D_m\mathcal J(G(\widetilde{u}),\widetilde{u}) \circ\left[ D_uG(\widetilde{u})\cdot h\right] + D_u\mathcal  J(G(\widetilde{u}),\widetilde{u})\cdot h\nonumber\\
&=\int_{\Omega_T} \widetilde{u}\cdot h \ dx\ dt + \int_{\Omega_T} \nabla \widetilde{u} \cdot \nabla h \ dx+ \int_\Omega \left(\widetilde{m}(x,T)-m_\Omega\right)\cdot z(x,T)\ dx\nonumber\\
&\hspace{0.1in} + \int_{\Omega_T}  \big(\widetilde{m}(x,t)-m_d(x,t)\big)\cdot z \ dx \ dt,	 
\end{flalign}
where $z=D_uG(\widetilde{u})\cdot (u-\widetilde{u})\in W^{1,2}(0,T;H^3(\Omega),H^1(\Omega))$ is a unique regular solution of the linearized system \eqref{LS2} with $h=u-\widetilde{u}.$ The main idea here is to express the last two integrals of \eqref{FD1} by using the weak solutions (Theorem \ref{TAS}) of the adjoint system \eqref{AP}.  

By testing \eqref{AP} with  $z$ and doing space integration by parts yield that
\begin{align}
& -\int_0^T\big\langle\phi_t, z\big\rangle\ dt   + \int_{\Omega_T} \nabla \phi \cdot \nabla z \ dx\ dt +\int_{\Omega_T} \nabla (\phi \times \widetilde{m})\cdot \nabla z \ dx\ dt- \int_{\Omega_T} (\Delta \widetilde{m} \times \phi)\cdot z \ dx\ dt\nonumber\\
	&\hspace{0.1in}  - \int_{\Omega_T} (\widetilde{u} \times \phi)\cdot z \ dx\ dt- 2\int_{\Omega_T}   (\widetilde{m}\cdot \phi)(\nabla \widetilde{m}\cdot \nabla z) \ dx\ dt-\int_{\Omega_T} |\nabla \widetilde{m}|^2 \phi \cdot z \ dx\ dt \label{FOE1}\\
	&\hspace{0.1in}  -\int_{\Omega_T}  \big((\phi \times \widetilde{m})\times \widetilde{u}\big)\cdot z \ dx\ dt - \int_{\Omega_T} \big( \phi \times (\widetilde{m} \times \widetilde{u})\big)\cdot z  \ dx\ dt=	\int_{\Omega_T}  \big(\widetilde{m}-m_d\big)\cdot z \ dx \ dt. \nonumber
\end{align}
Since $z\in H^1(0,T;H^1(\Omega))$ and $\phi\in H^1(0,T;H^1(\Omega)^*),$ time integrating by parts leads to the identity 
\begin{eqnarray}\label{TIP}
 \int_{0}^{T} \big( z_t,\phi\big) \ dt+\int_0^T \big\langle\phi_t, z\big\rangle \ dt= \int_\Omega \left(\widetilde{m}(x,T)-m_\Omega\right)\cdot z(x,T)\ dx.
\end{eqnarray}
On the other hand, testing the linearized system \eqref{LS2} with $\phi$ and integrating by parts, we get 
\begin{eqnarray}
	\int_{0}^{T} \big( z_t,\phi\big) \ dt\! &=&\!- \int_{\Omega_T}\nabla z \cdot \nabla \phi \ dx\ dt+ 2\int_{\Omega_T}   (\widetilde{m}\cdot \phi)\big(\nabla \widetilde{m}\cdot \nabla z\big) \ dx\ dt+\int_{\Omega_T} |\nabla \widetilde{m}|^2 z \cdot \phi \ dx\ dt\nonumber\\
	&&+\int_{\Omega_T}(z \times \Delta \widetilde{m})\cdot \phi\ dx\ dt+\int_{\Omega_T}  (\widetilde{m}\times \Delta z )\cdot \phi\ dx\ dt+\int_{\Omega_T}(z \times \widetilde{u})\cdot \phi \ dx\ dt\nonumber\\
	&& - \int_{\Omega_T}\big( z \times (\widetilde{m}\times \widetilde{u})\big)\cdot \phi \ dx\ dt-\int_{\Omega_T} \big( \widetilde{m}\times (z \times \widetilde{u})\big)\cdot \phi \ dx\ dt\nonumber\\
	&&+\int_{\Omega_T}(\widetilde{m}\times h)\cdot \phi \ dx\ dt- \int_{\Omega_T}\big(\widetilde{m}\times (\widetilde{m}\times h)\big)\cdot \phi \ dx\ dt:=\sum_{i=1}^{10}J_i. \label{FOE2}
\end{eqnarray}
Notice that the following cross product identities for $J_5$ and $J_8$ hold through Lemma \ref{CPP}:
\begin{eqnarray*}
\int_{\Omega_T}  (\widetilde{m}\times \Delta z )\cdot \phi\ dx\ dt=\int_{\Omega_T}  (\phi\times\widetilde{m})\cdot \Delta z\ dx\ dt=-\int_{\Omega_T} \nabla (\phi\times\widetilde{m})\cdot\nabla z \ dx\ dt,\\
-\int_{\Omega_T} \big( \widetilde{m}\times (z \times \widetilde{u})\big)\cdot \phi \ dx\ dt = \int_{\Omega_T} ( \widetilde{m}\times\phi) \cdot (z \times \widetilde{u}) \ dx\ dt =\int_{\Omega_T}  \big((\phi \times \widetilde{m})\times \widetilde{u}\big)\cdot z \ dx\ dt.
\end{eqnarray*} 
Further, applying a similar  vector identities  for the integrals $J_4,J_6$ and $J_7$ of \eqref{FOE2} and substitute \eqref{FOE2} into \eqref{TIP}. Then substituting \eqref{FOE1} and \eqref{TIP} into \eqref{FD1}, one can notice that most of the terms  cancel and arrive at the following:
$$	D_u \mathcal{I}(\widetilde{u})\cdot h= \int_{\Omega_T} \widetilde{u}\cdot h \ dx\ dt + \int_{\Omega_T} \nabla \widetilde{u}\cdot  \ \nabla h \ dx\  dt   + \int_{\Omega_T} \Big(  \widetilde{m}\times h- \widetilde{m}\times (\widetilde{m} \times h) \Big)\cdot \phi  \ dx \ dt.$$
Using \eqref{LI}, and again applying the property $(a \times b)\cdot c=b \cdot (c \times a)$ for the last integral, we get
\begin{eqnarray*}
\lefteqn{\int_{\Omega_T} \widetilde{u}\cdot \big(u - \widetilde{u}\big)  \ dx\ dt + \int_{\Omega_T} \nabla \widetilde{u} \cdot \nabla  \big(u - \widetilde{u}\big)  \ dx\  dt }
\\&&+\int_{\Omega_T}  \Big(  (\phi \times \widetilde{m})+  \widetilde{m} \times (\phi \times \widetilde{m}) \Big)\cdot \big( u-\widetilde{u} \big)\ dx \ dt \ \geq 0, \ \ \forall \ u \in \mathcal{U}_{ad}.
\end{eqnarray*}
Hence the proof. 
\end{proof}

\end{document}